\documentclass[11pt]{article}
\usepackage{amsmath, amsfonts, amssymb}
\usepackage{graphicx, amsthm}
\usepackage{cite}
\usepackage{color}
\usepackage{enumerate}
\usepackage{amsmath}
\usepackage{amsthm}
\usepackage{amssymb}
\usepackage{mathtools}
\usepackage{cleveref}

\providecommand{\U}[1]{\protect\rule{.1in}{.1in}}

\newtheorem{theorem}{Theorem}[section]
\newtheorem{lemma}[theorem]{Lemma}

\newtheorem{corollary}[theorem]{Corollary}
\newtheorem{proposition}[theorem]{Proposition}
\theoremstyle{definition}
\newtheorem{definition}[theorem]{Definition}

\theoremstyle{remark}

\numberwithin{equation}{section}

\newcommand{\trans}{\operatornamewithlimits{\pitchfork
	}}

\usepackage[margin=1.25in, top=1.3in, bottom=1.3in]{geometry}
\usepackage[vlined,ruled]{algorithm2e}
\SetKwComment{Comment}{[}{]}

\def\R{{\bf R}}

\def\gph{{\rm gph}\,}

\def\dom{\textrm{dom}\,}
\def\dist{{\rm dist}}

\def\ri{\textrm{ri}\,}

\def\prox{{\rm prox}}
\def\epi{{\rm epi}\,}

\newcommand{\li}{\operatornamewithlimits{liminf}}
\newcommand{\argmin}{\operatornamewithlimits{argmin}}

\def\B{{\bf B}}

\renewcommand\footnotemark{}
\title{Error bounds, quadratic growth, and linear convergence of proximal methods }
\author{Dmitriy Drusvyatskiy\thanks{University of Washington, Department of Mathematics, 
		Seattle, WA 98195; www.math.washington.edu/{\raise.17ex\hbox{$\scriptstyle\sim$}}ddrusv. Research of Drusvyatskiy was partially supported by the AFOSR YIP award FA9550-15-1-0237.}\and Adrian S. Lewis\thanks{School of Operations Research and Information Engineering, Cornell University, Ithaca, New
		York, USA; http://people.orie.cornell.edu/{\raise.17ex\hbox{$\scriptstyle\sim$}}aslewis/. Research supported in part by National
		Science Foundation Grants DMS-1208338 and DMS-1613996.}}

\begin{document}
\date{}
\maketitle

\abstract{The proximal gradient algorithm for minimizing the sum of a smooth and
	a nonsmooth convex function often converges linearly even without
	strong convexity.  One common reason is that a multiple of the step
	length at each iteration may linearly bound the ``error'' -- the distance to the
	solution set.  We explain the observed linear convergence intuitively
	by proving the equivalence of such an error bound to a natural quadratic
	growth condition. Our approach generalizes to linear convergence analysis
	for proximal methods (of Gauss-Newton type) for minimizing compositions of
	nonsmooth functions with smooth mappings. We observe incidentally that
	short step-lengths in the algorithm indicate near-stationarity, suggesting
	a reliable termination criterion.}

\section{Introduction}
Under favorable conditions, many fundamental optimization algorithms converge linearly:  the distance of the iterates to the optimal solution set (the ``error'') is bounded by a decreasing geometric sequence.  Classical optimization literature highlights how quadratic growth properties of the objective function, typically guaranteed through second-order optimality conditions, ensure such linear convergence.  Central examples traditionally include the method of steepest descent for smooth minimization \cite[Theorem 3.4]{ARAG} and, more abstractly, the proximal point method for nonsmooth convex problems \cite[Theorem 2, Proposition 7]{prox_rock}.

More recent techniques, originally highlighted in the work of Luo and Tseng~\cite{error_conv}, postulate that the step length at each iteration of the algorithm linearly bounds the error.  Such ``error bounds'' are commonly used in the analysis of first-order methods for strongly convex functions, popular in modern applications such as machine learning and high-dimensional statistics, including in particular the proximal gradient method and its variants; see for example Nesterov~\cite{Nes04b} and Beck-Teboulle~\cite{fista}.
Convergence analysis based only on the error bound property is appealingly simple even without strong convexity, but the underlying assumption on the optimization problem is opaque at least at first sight.

Some recent developments have focused on linear convergence guarantees based on more intuitive, geometric properties, akin to the classical quadratic growth condition.  An interesting example is \cite{nest_nonconv}.  Our aim here is to take a thorough and systematic approach, that is generalizable to problems with more complex structure.  Our aim is to show, in several interesting contemporary optimization frameworks, the equivalence between, on the one hand, the intuitive notion of quadratic growth of the objective function away from the set of minimizers, and on the other hand, the powerful analytic tool furnished by an error bound. Rockafellar already foreshadowed this possibility with his original analysis of the proximal point method \cite{prox_rock}.  We extend that relationship here to the {\em proximal gradient method} for problems 
$$\min_x~ f(x)+g(x),$$
with $g$ convex and $f$ convex and smooth, and more generally to the {\em prox-linear algorithm} (a variant of Gauss-Newton) for convex-composite problems
$$\min_x~ g(x)+h(c(x)),$$
where $g$ is a extended-real-valued closed convex function, $h$ is a finite-valued convex function, and $c$ is a smooth mapping. Acceleration strategies for the prox-linear algorithm have recently appeared in \cite{prox_lin_accel}.
In parallel, we show how the error bound property quickly yields linear convergence guarantees.
In essence, our analysis depends on viewing these two methods as approximations of the original proximal point algorithm -- a perspective of an independent interest.  Our assumptions are mild:  we rely primarily on a natural strict complementarity condition. In particular, we simplify and extend some of the novel convergence guarantees established in the recent preprint \cite{so_error} for the prox-gradient method.\footnote{While finalizing a first version of this work, the authors became aware of a concurrent, independent and nicely complementary approach \cite{ting_kei}, based on a related calculus of Kurdyka-\L ojasiewicz exponents.}

The iterative algorithms we consider assign a ``gradient-like'' step to each potential iterate, as in the analysis of proximal methods in \cite[Section 2.1.5]{Nes04b};  the step length is zero at stationary points and otherwise serves as a surrogate measure of optimality.  For steepest descent, the step is simply a multiple of the negative gradient, for the proximal point method it is determined by a subdifferential relationship, while the prox-gradient and prox-linear  methods combine the two.  In the language of variational analysis, the existence of an error bound is exactly ``metric subregularity'' of the gradient-like mapping; see Dontchev-Rockafellar~\cite{imp}. We will show that subregularity of the gradient-like mapping is equivalent to subregularity of the subdifferential of the objective function itself, thereby allowing us to call on extensive literature relating the quadratic growth of a function to metric subregularity of its subdifferential \cite{tilt_other, crit_semi, tilt, klat_kum_book}.  Given the generality of these techniques, we expect that the approach we describe here, rooted in understanding linear convergence through quadratic growth, should extend broadly. We note, in particular, some parallel developments influenced by the first version of this work \cite{error_bound_arxiv}, in the recent manuscript \cite{zhang_rest}.

When analyzing the prox-linear algorithm, we encounter a surprise. The error bound condition yields a linear convergence rate that is an order of magnitude worse than the natural rate for the prox-gradient method in the convex setting. The difficulty is that in the nonconvex case, the ``linearizations'' used by the method do not {\em lower-bound} the objective function. Nonetheless, we show that the method does converge with the natural rate if the objective function satisfies the stronger condition of quadratic growth that is uniform with respect to tilt-perturbations -- a property equivalent to the well-studied notions of tilt-stability \cite{tilt} and strong metric regularity of the subdifferential \cite{imp}. Concretely, these notions reduce to strong second-order sufficient conditions in nonlinear programming \cite{tilt_with_rock}, whenever the active gradients are linearly independent.

An important byproduct of our analysis, worthy of independent interest, relates the step-lengths taken by the prox-linear method to near-stationarity of the objective function at the iterates. Therefore, short step-lengths can be used to terminate the scheme, with explicit guarantees on the quality of the final solution.

We end by studying to what extent the tools we have developed generalize to composite optimization where the outer function  $h$  may be neither convex nor continuous -- an arena of growing recent interest (e.g. \cite{prx_lin}).  While considerably more technical, key ingredients of our analysis extend to this very general setting.

The outline of the manuscript is as follows. \Cref{sec:prelim} briefly records some elementary preliminaries. \Cref{sec:lin_con} contains a detailed analysis of linear convergence of the prox-gradient method for convex functions through the lens of error bounds and quadratic growth. In \cref{sec:quad_gr_error}, we show that quadratic growth holds in concrete applications under a mild condition of dual strict complementarity; our analysis aims to illuminate and extend some of the results in \cite{so_error} by dispensing with strong convexity of component functions. \Cref{sec:prox_lin_alg} is dedicated to the local linear convergence of the prox-linear algorithm for minimizing compositions of convex functions with smooth mappings. \Cref{sec:conv_rate_tilt} explains how a uniform notion of quadratic growth implies linear convergence of the prox-linear method with the natural rate. \Cref{sec:prox_wo_con} explains the resulting consequences for the prox-gradient method when the smooth component is not convex. The final \cref{prox_lin_full_gen} shows the equivalence between the error bound property of the prox-linear map and subdifferential subregularity  when both the component functions $g$ and $h$ may be infinite-valued and non-convex.

\section{Preliminaries}\label{sec:prelim}
Unless otherwise stated, we follow the terminology and notation of \cite{VA,Nes04b,imp}. Throughout $\R^n$ will denote an $n$-dimensional Euclidean space with inner-product $\langle \cdot, \cdot\rangle$ and corresponding norm $\|\cdot\|$. The closed unit ball will be written as ${\bf B}$, while the open ball of radius $r$ around a point $x$ will be denoted by $B_r(x)$.
For any set $Q\subset\R^n$, we define the {\em distance function} 
$$\dist(x;Q):=\inf_{z\in Q}\|z-x\|.$$

The functions we consider will take values in the extended real line $\overline{\R}:=\R\cup\{\pm \infty\}$. The {\em domain} and the {\em epigraph} of a function $f\colon\R^n\to\overline{\R}$ are defined by 
\begin{align*}
\dom f&:=\{x\in\R^n: f(x)<+\infty\},\\
\epi f &:=\{(x,r)\in \R^n\times \R: f(x)\leq r\}
,
\end{align*}
respectively.
We say that $f$ is {\em closed} if the inequality $\li_{x\to\bar x} f(x)\geq f(\bar x)$ holds for any point $\bar{x}\in \R^n$. The symbol $[f\leq \nu]:=\{x:f(x)\leq \nu\}$ will denote the $\nu$-sublevel set of $f$. For any set $Q\subset \R^n$, the indicator function $\delta_Q\colon\R^n\to\overline\R$ evaluates to zero on $Q$ and to $+\infty$ elsewhere.

The Fenchel conjugate of a convex function $f\colon\R^n\to\overline \R$ is the closed convex function $f^{\star}\colon\R^n\to\overline{\R}$ defined by 
$$f^{\star}(y)=\sup_{x}\,\{\langle y,x\rangle -f(x)\}.$$ The {\em subdifferential} of a convex function $f$ at a point $x$, denoted by $\partial f(x)$, is the set consisting of all vectors $v\in \R^n$ satisfying $f(z)\geq f(x)+\langle v,z-x\rangle$ for all  $z\in\R^n$.
For any function $f\colon\R^n\to\overline{\R}$ and a real number $t>0$, we define the {\em Moreau envelope} $$g^t(x):=\min_y \left\{g(y)+\frac{1}{2t}\|y-x\|^2 \right\},$$
and the {\em proximal mapping} $$\prox_{tg}(x):=\argmin_{y\in\R^n}\, \{g(y)+\frac{1}{2t}\|y-x\|^2\}.$$
The proximal map $x\mapsto \prox_{tg}(x)$ is always 1-Lipschitz continuous.

A set-valued mapping $F\colon\R^n\rightrightarrows \R^m$ is a mapping assigning to each point $x\in \R^n$ the subset $F(x)$ of $\R^m$. The {\em graph} of such a mapping is the set
$$\gph F:=\{(x,y)\in \R^n\times\R^m: y\in F(x)\}.$$
The inverse map $F^{-1}\colon \R^m\rightrightarrows\R^n$
is defined by setting $F^{-1}(y)=\{x:y\in F(x)\}$.
Every mapping $F\colon\R^n\rightrightarrows\R^n$ obeys the identity \cite[Lemma 12.14]{VA}:
\begin{equation}\label{prox_formula}
(I+F)^{-1}+(I+F^{-1})^{-1}=I.
\end{equation}
Note that for any convex function $g$ and a real $t>0$, equality $\prox_{tg}=(I+t\partial g)^{-1}$ holds.

\section{Linear convergence of the prox-gradient method}\label{sec:lin_con}

To motivate the discussion, consider the optimization problem 
\begin{equation}\label{eqn:main_problem}
\min_{x}~ \varphi(x):=f(x)+g(x)
\end{equation}
where $g\colon\R^n\to\overline{\R}$ is a closed convex function and $f\colon\R^n\to\R$ is a 
convex $C^1$-smooth function with a $\beta$-Lipschitz continuous gradient:
$$\|\nabla f(x)-\nabla f(y)\|\leq \beta \|x-y\|\qquad \textrm{ for all } x,y\in \R^n.$$
The {\em proximal gradient method} is the recurrence
$$x_{k+1}=\argmin_{x\in \R^n}~ f(x_k)+\langle \nabla f(x_k),x-x_k\rangle+\frac{1}{2t}\|x-x_k\|^2+g(x),$$
where the constant $t>0$ is appropriately chosen. More succinctly, the method simply iterates the steps  $$x_{k+1}=\prox_{tg}(x_k-t\nabla f(x_k)).$$
In order to see the parallel between the proximal gradient method and classical gradient descent for smooth minimization, it is convenient to rewrite the recurrence yet again as $x_{k+1}=x_k-t\mathcal{G}_{t}(x_k)$
where 
$$\mathcal{G}_t(x):=t^{-1}\Big(x-\prox_{t g}(x-t\nabla f(x))\Big)$$
is the {\em prox-gradient mapping}.
In particular, equality $\mathcal{G}_t(x)=0$ holds if and only if $x$ is optimal for $\varphi$.

Let $S$ be the set of minimizers of $\varphi$ and let $\varphi^*$ be the minimal value of $\varphi$.
Supposing now $t\leq \beta^{-1}$, the following two inequalities are standard \cite[Theorem 2.2.7, Corollary 2.2.1]{Nes04b} and \cite[Lemma~2.3]{fista}: 
\begin{align}
\varphi(x_k)-\varphi(x_{k+1})&\geq \frac{1}{2\beta}\|\mathcal{G}_t(x_k)\|^2\label{eqn:descent},\\
\varphi(x_{k+1})-\varphi^*&\leq \langle \mathcal{G}_t(x_k),x_k-x^*\rangle-\frac{1}{2\beta}\|\mathcal{G}_t(x_k)\|^2\label{eqn:decrease}.
\end{align}
Here $x^*$ denotes an arbitrary element of $S$.
Hence equation~\eqref{eqn:decrease} immediately implies
$$\varphi(x_{k+1})-\varphi^*\leq \|\mathcal{G}_t(x_k)\|^2\left(\frac{\|x^*-x_k\|}{\|\mathcal{G}_t(x_k)\|}-\frac{1}{2\beta}\right).$$
Defining $\gamma_k:=\frac{\|x^*-x_k\|}{\|\mathcal{G}_t(x_k)\|}$ and using inequality \eqref{eqn:descent}, along with some trivial algebraic manipulations, yields the geometric decrease guarantee
\begin{equation}\label{eq:lin_con}
\varphi(x_{k+1})-\varphi^*\leq \left(1-\frac{1}{2\beta\gamma_k}\right)(\varphi(x_{k})-\varphi^*).
\end{equation}
Hence if the quantities $\gamma_k$ are bounded for all large $k$, asymptotic Q-linear convergence in function values is assured. This observation motivates the following definition, originating in \cite{error_conv}.

\begin{definition}[Error bound condition]\label{defn:error_bound1}
	{\rm
		Given real numbers $\gamma,\nu>0$,	
		we say that the {\em error bound condition holds with parameters $(\gamma,\nu)$} if the inequality
		$$\dist(x,S)\leq \gamma \|\mathcal{G}_t(x)\|\quad \textrm{ is valid for all }\quad x\in [\varphi\leq \varphi^*+ \nu].$$
	}
\end{definition}

Hence we have established the following.
\begin{theorem}[Linear convergence]\label{thm:lin_conv}
	Suppose the error bound condition holds with parameters $(\gamma,\nu)$. Then the proximal gradient method with $t\leq \beta^{-1}$ satisfies 
	$\varphi(x_k)-\varphi^*\leq \epsilon$ after at most
	\begin{equation}\label{eq:lin-rate}
	k\leq \frac{\beta}{2\nu}\dist^2(x_0,S)+2\beta\gamma\ln\left(\frac{\varphi(x_0)-\varphi^*}{\epsilon}\right)\qquad\textrm{iterations}.
	\end{equation}
	Moreover, if the iterates $x_k$ have some limit point $x^*$, then there exists an index $r$ such that the inequality
	$$\|x_{r+k}-x^*\|^2\leq \left(1-\frac{1}{2\beta\gamma}\right)^{k}C\cdot (\varphi(x_r)-\varphi^*),$$
	holds for all $k\geq 1$, where we set $C:=\frac{2}{\beta(1-\sqrt{1-(2\beta\gamma)^{-1}})^2}$.
\end{theorem}
\begin{proof}
	From the the standard sublinear estimate $\varphi(x_k)-\varphi^*\leq \frac{\beta\cdot\dist^2(x_0,S)}{2k}$ (see e.g. \cite[Theorem~3.1]{fista}), we deduce that after $k\leq \frac{\beta}{2\nu}\dist^2(x_0,S)$ iterations the inequality $\varphi(x_k)-\varphi^*\leq \nu$ holds. The second summand in inequality \eqref{eq:lin-rate} is then immediate from the linear rate \eqref{eq:lin_con} and the fact that the values $\varphi(x_k)$ decrease monotonically. 
	
	Now suppose that $x^*$ is a limit point of $x_k$. Note that if an iterate $x_r$ lies in the set $[\varphi\leq \varphi^*+ \nu]$, then we have
	\begin{align*}
	\|x_{r+k}&-x^*\|\leq \sum^{\infty}_{i=r+k}\|x_{i}-x_{i+1}\|\leq \sqrt{2/\beta}\sum^{\infty}_{i=r+k}\sqrt{\varphi(x_i)-\varphi(x_{i+1})} \\
	&\leq\sqrt{2/\beta}\sqrt{\varphi(x_r)-\varphi^*}\sum^{\infty}_{i=r+k} \left(1-\frac{1}{2\beta\gamma}\right)^{\frac{i-r}{2}}\leq
	\left(1-\frac{1}{2\beta\gamma}\right)^{k/2}D\sqrt{\varphi(x_r)-\varphi^*},
	\end{align*}
	where we set $D:=\frac{\sqrt{2}}{\sqrt{\beta}(1-\sqrt{1-(2\beta\gamma)^{-1}})}$.
	Squaring both sides, the result follows.
\end{proof}

Convergence guarantees of Theorem~\ref{thm:lin_conv} are expressed in terms of the error bound parameters $(\gamma,\nu)$ -- quantities not stated in terms of the initial data of the problem, $f$ and $g$. Indeed, the error bound condition is a property of the prox-gradient mapping $\mathcal{G}_t(x)$, a nontrivial object to understand. In contrast, in the current work we will show that the error bound condition is simply equivalent to the objective function $\varphi$ growing quadratically away from its minimizing set $S$ -- a familiar, transparent, and largely classical property in nonsmooth optimization.

To gain some intuition, consider the simplest case $g=0$. Then the prox-gradient method reduces to gradient descent $x_{k+1}=x_k-t\nabla f(x_k)$. Suppose now that $f$ grows quadratically (globally) away from its minimizing set, meaning there is a real number $\alpha >0$ such that
$$f(x)\geq f^*+\frac{\alpha}{2}\dist^2(x,S)\qquad \textrm{ for all } x\in\R^n.$$
Notice this property is weaker than strong convexity even for $C^1$-smooth functions; e.g $f(x)=(\max\{|x|-1,0\})^2$. Then convexity implies 
$$\frac{\alpha}{2} \dist^2(x,S)\leq f(x)-f^*\leq \langle \nabla f(x),x^*-x\rangle\leq \|\nabla f(x)\|\|x-x^*\|.$$
Thus the error bound condition holds with parameters $(\gamma,\nu)=(\frac{2}{\alpha},\infty)$, and the complexity bound of Theorem~\ref{thm:lin_conv} becomes 
$k\leq \frac{4\beta}{\alpha}\ln\left(\frac{\varphi(x_0)-\varphi^*}{\epsilon}\right)$. This is the familiar linear rate of gradient descent (up to a constant).

Our goal is to elucidate the quantitative relationship between quadratic growth and the error bound condition in full generality. The strategy we follow is very natural; we will interpret the proximal gradient method as an approximation to the true proximal point algorithm $y_{k+1}=\prox_{t\varphi}(y_k)$ on the function $\varphi=f+g$, and show a linear relationship between the corresponding step sizes (Theorem~\ref{thm:approx}). This will allows us to ignore the linearization appearing in the definition of the proximal gradient method and focus on the relationship between quadratic growth of $\varphi$, properties of the mapping $\prox_{t\varphi}$, and of the subdifferential $\partial\varphi$ (Theorems~\ref{thm:subdif_quad} and \ref{thm:prox_sub}). We believe this interpretation of the proximal gradient method is of interest in its own right.



The following is a central result we will need. It establishes a relationship between quadratic growth properties and a ``global error bound property'' of the function $x\mapsto d(0,\partial \varphi(x))$. Variants of this result have appeared in 
\cite[Theorem 3.3]{fran_sub}, \cite[Theorem 6.1]{nonlin_met}, \cite[Theorem 4.3]{crit_semi}, \cite[Theorem 3.1]{tilt_other}, and \cite{klat_kum_book}. 
\begin{theorem}[Subdifferential error bound and quadratic growth]\label{thm:subdif_quad}
	Consider a closed convex function $h\colon\R^n\to\overline{\R}$ with minimal value $h^*$ and let $S$ be its set of minimizers. Consider the conditions
	\begin{equation}\label{eqn:quad_growth}
	h(x)\geq h(\bar{x})+\frac{\alpha}{2}\cdot \dist^2(x; S) \qquad\textrm{ for all }x\in [ h\leq h^*+\nu]
	\end{equation}
	and
	\begin{equation}\label{eqn:reg_growth}
	\dist(x; S)\leq L\cdot \dist(0;\partial h(x)) \qquad\textrm{ for all }x\in [h\leq h^*+\nu].
	\end{equation}
	If condition \eqref{eqn:quad_growth} holds, then so does condition \eqref{eqn:reg_growth} with $L=2\alpha^{-1}$. Conversely, condition \eqref{eqn:reg_growth} implies condition \eqref{eqn:quad_growth} with any $\alpha\in (0,\frac{1}{L}]$. 
\end{theorem}

The proof of the implication $\eqref{eqn:quad_growth} \Rightarrow \eqref{eqn:reg_growth}$ is identical to the proof of the analogous implication in \cite[Theorem 3.3]{fran_sub}; the proof of the implication $\eqref{eqn:reg_growth}\Rightarrow \eqref{eqn:quad_growth}$ is the same as that of \cite[Theorem 4.3]{crit_semi},\cite[Theorem 3.1]{tilt_other}. Hence we omit the arguments.

Given the equality $\prox_{th}=(I+t\partial h)^{-1}$, it is clear that the subdifferential error bound condition 
\eqref{eqn:reg_growth} is related to an analogous property of the proximal mapping. This is the content of the following elementary result.


\begin{theorem}[Proximal and subdifferential error bounds]\label{thm:prox_sub}
	Consider a closed convex function $h\colon\R^n\to\overline{\R}$ with minimal value $h^*$ and let $S$ be its set of minimizers. Consider the conditions
	\begin{equation}\label{eqn:subreg_sub}
	\dist(x; S)\leq L\cdot \dist(0;\partial h(x)) \qquad\textrm{ for all } x\in [ h\leq h^*+\nu].
	\end{equation}
	and 
	\begin{equation}\label{eqn:subreg_prox}
	\dist(x; S)\leq \widehat{L}\cdot t^{-1}\|x-\prox_{th}(x)\| \qquad\textrm{ for all }x\in [ h\leq h^*+\nu].
	\end{equation}
	If condition \eqref{eqn:subreg_sub} holds, then so does condition \eqref{eqn:subreg_prox} with $\widehat{L}=L+t$. Conversely, condition \eqref{eqn:subreg_prox} implies condition \eqref{eqn:subreg_sub} with $L=\widehat{L}$.
\end{theorem}
\begin{proof}
	Suppose condition~\eqref{eqn:subreg_sub} holds and consider a point $x\in [ h\leq h^*+\nu]$. Then clearly the inequality
	$h(\prox_{th}(x))\leq h(x)\leq h^*+\nu$ holds. Taking into account the inclusion $t^{-1}(x-\prox_{th}(x))\in \partial h(\prox_{th}(x))$, we obtain
	\begin{align*}
	\dist(x,S)&\leq \|x-\prox_{th}(x)\|+\dist(\prox_{th}(x),S)\\ &\leq\|x-\prox_{th}(x)\| +L\cdot \dist(0;\partial h(\prox_{th}(x)))\\
	&\leq (1+t^{-1}L)\|x-\prox_{th}(x)\|,
	\end{align*}
	as claimed.
	Conversely suppose condition~\eqref{eqn:subreg_prox} holds and fix a point $x\in [ h\leq h^*+\nu]$. Then for any subgradient $v\in\partial h(x)$, equality $\prox_{th}(x+tv)=x$ holds. Hence, we obtain
	\begin{align*}
	\dist(x,S)&\leq \widehat{L}t^{-1} \|x-\prox_{th}(x)\|\leq \widehat{L}t^{-1} \left\|\prox_{th}(x+tv)-\prox_{th}(x)\right\|\leq \widehat{L}\|v\|,
	\end{align*}
	where we have used the fact that the proximal mapping is 1-Lipschitz continuous.
	Since the subgradient $v\in\partial h(x)$ is arbitrary, the result follows. 
\end{proof}

The final step is to relate the step sizes taken by the proximal gradient and the proximal point methods. The ensuing arguments are best stated in terms of monotone operators. To this end, observe that our running problem \eqref{eqn:main_problem} is equivalent to solving the inclusion 
$$0\in \nabla f(x)+\partial g(x).$$ 
More generally, consider monotone operators $F\colon\R^n\to\R^n$ and $G\colon\R^n\rightrightarrows\R$, meaning that $F$ and $G$ satisfy the inequalities $\langle v_1-v_2,x_1-x_2\rangle \geq 0$ and $\langle F(x_1)-F(x_2),x_1-x_2 \rangle \geq 0$ for all $x_i\in \R^n$ and $v_i\in G(x_i)$ with $i=1,2$. We now further assume that $G$ is {\em maximal monotone}, meaning that the graph $\gph G$
is not a proper subset of the graph of any other monotone operator.
Along with the operator $G$ and a real $t >0$, we associate the {\em resolvent} 
$$\prox_{tG}:=(I+tG)^{-1}.$$
The mapping $\prox_{tG}\colon\R^n\to\R^n$ is then single-valued and nonexpansive ($1$-Lipschitz continuous \cite[Theorem 12.12]{VA}). We aim to solve the inclusion 
$$0\in \Phi(x):=F(x)+G(x)$$
by the {\em Forward-Backward algorithm}
$$x_{k+1}=\prox_{tG}(x_k-tF(x_k)),$$
Equivalently we may write $x_{k+1}=x_k-t\mathcal{G}_t(x_k)$ where 
$\mathcal{G}_t(x)$ is the {\em prox-gradient} mapping
$$\mathcal{G}_t(x):=\frac{1}{t}\left(x-\prox_{tG}(x-t F(x)\right).$$
Setting $F=\nabla f$, $G:=\partial g$, $\Phi:=\partial \varphi$ recovers the proximal gradient method for the problem \eqref{eqn:main_problem}. 

The following key result shows that the step lengths of the Forward-Backward algorithm and those taken by the proximal point algorithm $z_{k+1}=\prox_{t\Phi}(z_k)$ are proportional.

\begin{theorem}[Step-lengths comparison]\label{thm:approx} 
	Consider two maximal monotone operators $G\colon\R^n\rightrightarrows\R^n$ and $\Phi\colon\R^n\rightrightarrows\R^n$, with the difference $F:=\Phi-G$ that is single-valued.
	Then the inequality 
	\begin{equation}\label{eqn:easy}
	\|\mathcal{G}_t(x)\|\leq \dist\left(0; \Phi(x)\right)\qquad \textrm{ holds.}
	\end{equation} 
	Supposing that $F$ is in addition $\beta$-Lipschitz continuous, the inequalities hold:
	\begin{equation}\label{eqn:main_ineq}
	(1- \beta t)\cdot \|\mathcal{G}_t(x)\|\leq \|t^{-1}\left(x-\prox_{t\Phi}(x)\right)\|\leq (1+\beta t)\cdot\|\mathcal{G}_t(x)\|
	\end{equation}
\end{theorem}
\begin{proof}
	Fix a point $x\in\R^n$ and a vector $v\in \Phi(x)$. Then clearly the inclusion
	$$(x-t F(x))+t v\in x+t G(x) \qquad \textrm{ holds},$$
	or equivalently
	$x=\prox_{t G}((x-tF(x))+t v).$ Since the proximal mapping is nonexpansive, we deduce 
	$$t\|\mathcal{G}_t(x)\|=\|x-\prox_{t G}(x-tF(x))\|\leq t\|v\|.$$  Letting $v$ be the minimal norm element of $\Phi(x)$, we deduce the claimed inequality $\|\mathcal{G}_t(x)\|\leq \dist\left(0; \Phi(x)\right)$. 
	
	Now suppose that $F$ is $\beta$-Lipschitz continuous. Consider a point $x\in\R^n$ and define $z:=\mathcal{G}_t(x)$. Observe the chain of equivalences:
	\begin{align*}
	z= \mathcal{G}_t(x) &\quad\Leftrightarrow\quad tz= x-\prox_{tG}(x-tF(x))\\
	&\quad\Leftrightarrow\quad x-tF(x)\in (x-tz) +tG(x-tz)\\
	&\quad\Leftrightarrow\quad x+t\left(F(x-tz)-F(x)\right) \in (I+t\Phi)(x-tz)
	\end{align*}
	Define now the vector $w=F(x-tz)-F(x)$ and note $\|w\|\leq  \beta t\|z\|$.
	Hence taking into account that resolvents are nonexpansive, we obtain
	$$x-tz = \prox_{t\Phi}(x+t w)\subset \prox_{t\Phi}(x)+t \|w\|\B,$$ and so deduce
	$$z\in \frac{1}{t}\left(x-\prox_{t\Phi}(x)\right) + \beta t\|z\|\B.$$
	Hence
	$$\Big|\|z\|-t^{-1}\|x-\prox_{t\Phi}(x)\|\Big| \leq \|z-t^{-1}\left(x-\prox_{t\Phi}(x)\right)\|\leq \beta t\|z\|.$$
	The two inequalities in \eqref{eqn:main_ineq} follow immediately.
\end{proof}

We now arrive at the main result of this section.
\begin{corollary}[Error bound and quadratic growth]\label{cor:err_bound_quad}
	Consider a closed, convex function $g\colon\R^n\to\overline{\R}$ and a $C^1$-smooth convex function $f\colon\R^n\to\R$ with $\beta$-Lipschitz continuous gradient. Suppose that the function $\varphi:=f+g$ has a nonempty set $S$ of minimizers and consider the following conditions: 
	\begin{itemize}
		\item (Quadratic growth)
		\begin{equation}\label{eqn:quad_growth_thm}
		\varphi(x)\geq \varphi^{\star}+\frac{\alpha}{2}\cdot \dist^2(x; S) \qquad\textrm{ for all }x\in [ \varphi\leq \varphi^*+ \nu]
		\end{equation}
		\item (Error bound condition) 
		\begin{equation}\label{eqn:err_bound_thm}
		\dist(x,S)\leq \gamma \|\mathcal{G}_t(x)\|\quad \textrm{ is valid for all }\quad x\in [\varphi\leq \varphi^*+ \nu],
		\end{equation}
	\end{itemize}
	Then property \eqref{eqn:quad_growth_thm} implies property \eqref{eqn:err_bound_thm} with $\gamma=(2\alpha^{-1}+t)(1+\beta t)$. Conversely, condition \eqref{eqn:err_bound_thm} implies condition \eqref{eqn:quad_growth_thm} with any $\alpha\in (0,\gamma^{-1})$.
\end{corollary}
\begin{proof}
	Suppose condition \eqref{eqn:quad_growth_thm} holds. Then for any $x\in [\varphi\leq \varphi^*+\nu]$, we deduce
	\begin{align*}
	\dist(x,S)&\leq 2\alpha^{-1}\cdot\dist(0;\partial \varphi(x))&\textrm{(Theorem~\ref{thm:subdif_quad})}\\
	&\leq (2\alpha^{-1}+t)\|t^{-1}(x-\prox_{t\varphi}(x))\| &\textrm{(Theorem~\ref{thm:prox_sub})}\\
	&\leq (2\alpha^{-1}+t)(1+\beta t)\|\mathcal{G}_t(x)\|& \textrm{(Inequality~\ref{eqn:main_ineq})}
	\end{align*}
	This establishes \eqref{eqn:err_bound_thm} with $\gamma=(2\alpha^{-1}+t)(1+\beta t)$. Conversely suppose 
	\eqref{eqn:err_bound_thm} holds. Then for any $x\in [\varphi\leq \varphi^*+\nu]$ we deduce  using Theorem~\ref{thm:approx} the inequality
	$\dist(x,S)\leq \gamma \|\mathcal{G}_t(x)\|\leq \gamma\cdot \dist(0,\partial \varphi(x))$. An application of Theorem~\ref{thm:subdif_quad} completes the proof.
\end{proof}

The following convergence result is now immediate from Theorem~\ref{thm:lin_conv} and Corollary~\ref{cor:err_bound_quad}. Notice that the complexity bound matches (up to a constant) the linear rate of convergence of the proximal gradient method when applied to strongly convex functions. 
\begin{corollary}[Quadratic growth and linear convergence]
	Consider a closed, convex function $g\colon\R^n\to\overline{\R}$ and a $C^1$-smooth function $f\colon\R^n\to\R$ with $\beta$-Lipschitz continuous gradient. Suppose that the function $\varphi:=f+g$ has a nonempty set $S$ of minimizers and that the quadratic growth condition holds: 
	\begin{equation*}
	\varphi(x)\geq \varphi^*+\frac{\alpha}{2}\cdot \dist^2(x; S) \qquad\textrm{ for all }x\in [ \varphi\leq \varphi^*+\nu].
	\end{equation*}
	Then the proximal gradient method with $t\leq \beta^{-1}$ satisfies 
	$\varphi(x_k)-\varphi^*\leq \epsilon$ after at most
	$$k\leq \frac{\beta}{2\nu}\dist^2(x_0,S)+12\cdot\frac{\beta}{\alpha}\ln\left(\frac{\varphi(x_0)-\varphi^*}{\epsilon}\right)\qquad\textrm{iterations}.$$	
\end{corollary}

\section{Quadratic growth in structured optimization}\label{sec:quad_gr_error}
Recently, the authors of \cite{so_error} proved that the error bound condition holds under very mild assumptions, thereby explaining asymptotic linear convergence of the proximal gradient method often observed in practice. In this section, we aim to use the equivalence between the error bound condition and quadratic growth,  established in Theorem~\ref{cor:err_bound_quad}, to streamline and illuminate the arguments in \cite{so_error}, while also extending their results to a wider setting. To this end, consider the problem
\begin{equation}\label{eqn:prim_prob}
\inf_x~ \varphi(x):=f(Ax)+g(x)
\end{equation}
where $f\colon\R^m\to\R$ is convex and $C^1$-smooth, $g\colon\R^n\to\overline{\R}$ is closed and convex, and $A\colon\R^n\to\R^m$ is a linear mapping. We assume that $g$ is proper, meaning that its domain is nonempty.
Consider now the Fenchel dual problem
\begin{equation}\label{eqn:dual_prob}
\sup_{y}\, -\Psi(y):=-f^{\star}(y)-g^{\star}(-A^Ty).
\end{equation}
By \cite[Corollary 31.2.1(a)]{con_book}, the optimal values of the primal \eqref{eqn:prim_prob} and of the dual \eqref{eqn:dual_prob} are equal, and the dual optimal value is attained. To make progress, we assume that the dual problem \eqref{eqn:dual_prob} admits a strictly feasible point: 
\begin{enumerate}
	\item\label{ass1} (dual nondegeneracy) $\quad 0\in A^T(\ri \dom f^{\star}) +\ri \dom g^{\star}$,
\end{enumerate}
Then by \cite[Corollary 31.2.1(b)]{con_book}, the optimal value of the primal \eqref{eqn:prim_prob} is also attained. From the Kuhn-Tucker conditions \cite[p. 333]{con_book}, any optimal solution of the dual coincides with $\nabla f(Ax)$ for any primal optimal solution $x$. In particular, the dual has a unique optimal solution and we will denote it by $\bar y$. Clearly, the inclusion $0\in \partial \Psi(\bar y)$ holds. We now assume the mildly stronger property:
\begin{enumerate}
	\setcounter{enumi}{1}
	\item\label{ass2} (dual strict complementarity) $\quad 0\in \ri \partial \Psi(\bar y)$.
\end{enumerate}
Taken together, these two standard conditions (dual nondegeneracy and dual strict complementarity) immediately imply 
\begin{equation}\label{eqn:strict_comp}
\begin{aligned}
0\in \ri \partial\Psi(\bar y) &= \ri\left(\partial f^{\star}(\bar y) -A\partial g^{\star}(-A^T\bar y)\right)\\
&=\ri \partial f^{\star}(\bar y) -A \left(\ri \partial g^\star(-A^T\bar y)\right),
\end{aligned}
\end{equation}
where the last equality follows for example from \cite[Theorem 6.6]{con_book}.

Let $S$ be the solution set of the primal problem \eqref{eqn:prim_prob}.
To elucidate the impact of the inclusion \eqref{eqn:strict_comp} on error bounds,
recall that we must estimate the distance $\dist(x,S)$ for an arbitrary point $x$.
To this end, the Kuhn-Tucker conditions again directly imply that $S$ admits the description
$$S=\partial g^{\star}(-A^T\bar{y})\cap A^{-1}\partial f^{\star}(\bar{y}).$$
The inclusion \eqref{eqn:strict_comp}, combined with \cite[Theorem 6.7]{con_book}, guarantees that the relative interiors of the two sets 
$\partial g^{\star}(-A^T\bar{y})$ and $A^{-1}\partial f^{\star}(\bar{y})$ meet and hence by for any compact set $\mathcal{X}\subset \R^n$ there exists a constant $\kappa\geq 0$ satisfying\footnote{This follows by applying \cite[Corollary 4.5]{baus_bor_lin_reg} first to the two sets $\partial g^{\star}(-A^T\bar{y})$ and $A^{-1}\partial f^{\star}(\bar{y})$, and then to the range of $A$ and $\partial f^{\star}(\bar{y})$.}
\begin{equation}\label{eqn:rel_int}
\dist(x,S)\leq \kappa \Big(\dist(x,\partial g^{\star}(-A^T\bar y))+\dist(Ax,\partial f^{\star}(\bar{y}))\Big)\qquad\qquad \textrm{ for all }x\in \mathcal{X}.
\end{equation}
This type of an inequality is often called linear regularity; see for example \cite[Corollary~4.5]{baus_bor_lin_reg}. The final assumption we need to deduce quadratic growth of $\varphi$, not surprisingly, is a quadratic growth condition on the individual functions $f$ and $g$ after tilt perturbations.

\begin{definition}[Firm convexity]
	{\rm
		A closed convex function $h\colon\R^n\to\overline{\R}$ is 
		{\em firmly convex relative to a vector $v\in \R^n$} if the tilted function 
		$h_v(x):=h(x)-\langle v,x\rangle$ satisfies the quadratic growth condition: for any compact set $\mathcal{X}\subset\R^n$ there is a constant $\alpha$ satisfying
		$$h_v(x)\geq (\inf h_v)+\frac{\alpha}{2}\dist^2(x,(\partial h_v)^{-1}(0))\qquad \textrm{ for all }x\in\mathcal{X}.$$
		We say that $h$ is {\em firmly convex} if $h$ is firmly convex relative to any vector $v\in\R^n$.}
\end{definition}

Note that not all convex functions are firmly convex; for example $f(x)=x^4$ is not firmly convex at $x=0$ relative to $v=0$.
We are now ready to prove the main theorem of this section; note that unlike in \cite{so_error}, we do not require strong convexity of the function $f$. This generalization is convenient since it allows to capture ``robust'' formulations where $f$ is a translate of the Huber penalty or  its asymmetric extensions.

\begin{theorem}[Quadratic growth in composite optimization]\label{thm:quad_comp}
	Consider a closed, convex function $g\colon\R^n\to\overline{\R}$ and a $C^1$-smooth convex function $f\colon\R^m\to\R$. Suppose that the sum $\varphi(x):=f(Ax)+g(x)$ has a nonempty set $S$ of minimizer and let  $\bar{y}$ be the optimal solution of the dual problem \eqref{eqn:dual_prob}. Suppose the  conditions hold:
	\begin{enumerate}
		\item (Compactness) The solution set $S$ is bounded.
		\item (Dual nondegeneracy and strict complementarity) Assumptions \ref{ass1} and \ref{ass2}.
		\item\label{it:quad_gr} (Quadratic growth of components) The functions $f$ and $g$ are firmly convex relative to $\bar{y}$ and $-A^T\bar y$, respectively.
	\end{enumerate}
	Then the error bound condition holds with some parameters $(\gamma,\nu)$.
	
\end{theorem}
\begin{proof}
	Since $S$ is compact, all sublevel sets of $\varphi$ are compact. Choose a number $\nu > 0$
	and set $\mathcal{X}:=[\varphi\leq \varphi^*+\nu]$ and $\mathcal{Y}= A(\mathcal{X})$. Let $\bar{x}\in S$ be arbitrary and note the equality $\bar y=\nabla f(A\bar x)$. Then observing that $A\bar{x}$ minimizes $f(\cdot)-\langle\bar y,\cdot\rangle$ and $\bar x$ minimizes $g(\cdot)+\langle A^T\bar y,\cdot \rangle$, property~\eqref{it:quad_gr} (Quadratic growth of components) guarantees that there exist constants $c,\alpha\geq 0$ such that 
	\begin{equation}\label{eqn:withy}
	f(y)\geq f(A\bar x)+\langle \bar y, y-A\bar x\rangle +\frac{c}{2}\dist^2(y,(\partial f)^{-1}(\bar y))\qquad \textrm{ for all } y\in \mathcal{Y},
	\end{equation}
	and
	$$g(x)\geq g(\bar x)+\langle -A^T\bar y, x-\bar x\rangle +\frac{\alpha}{2}\dist^2(x,(\partial g)^{-1}(-A^T\bar y))\qquad \textrm{ for all } x\in \mathcal{\mathcal{X}}.$$
	Letting $\kappa$ be the constant from \eqref{eqn:rel_int} and setting $y:=Ax$ in \eqref{eqn:withy}, we deduce 
	\begin{align*}
	\varphi(x)=f(Ax)+g(x)&\geq \Big(f(A\bar{x}) +\langle \bar y,Ax-A\bar{x}\rangle +\frac{c}{2}\dist^2(Ax,\partial f^{\star}(\bar y))\Big)+\\
	&+\Big( g(\bar x)+\langle -A^T\bar y,x-\bar x \rangle +\frac{\alpha}{2}\dist^2(x,\partial g^{\star}(-A^T\bar y)) \Big)\\
	&\geq \varphi(\bar{x})+\frac{\min\{\alpha,c\}}{4\kappa^2}\dist^2(x,S).
	\end{align*}
	This completes the proof.
\end{proof}

Notice that firm convexity requires a certain inequality to hold on compact sets $\mathcal{X}$, rather than on sublevel sets. In any case, firm convexity is intimately tied to error bounds. For example, analogously to Theorem~\ref{thm:subdif_quad}, one can show that $h$ is firmly convex relative to $v$ if and only if for any compact set $\mathcal{X}$ there exists a constant $L\geq 0$ satisfying
$$\dist(x,(\partial h)^{-1}(v))\leq L \cdot\dist(v,\partial h(x))\qquad \textrm{ for all }x\in \mathcal{X}.$$
Indeed this is implicitly shown in the proof of Theorem~\cite[Theorem 3.3]{fran_sub}, for example. Moreover, the same argument as in Theorem~\ref{thm:prox_sub} shows that  $h$ is firmly convex relative to $v$ if and only if for any compact set $\mathcal{X}$ there exists a constant $\widehat{L}\geq 0$ satisfying
$$\dist(x; S)\leq \widehat{L}\cdot \|t^{-1}(x-\prox_{th}(x))\| \qquad\textrm{ for all }x\in \mathcal{X}.$$

The class of firmly convex functions is large, including for example all strongly convex functions and polyhedral functions. More generally, all convex Piecewise Linear Quadratic (PLQ) functions \cite[Section 10.20]{VA} are firmly convex, since their subdifferential graphs are finite unions of polyhedra. Indeed, the subclass of affinely composed PLQ penalties \cite[Example 11.18]{VA} is ubiquitous in optimization. These are functions of the form
$$h(x):=\sup_{z \in Z}~ \langle Bx-b,z\rangle-\langle Az,z \rangle$$
where $Z$ is a polyhedron, $B$ is a linear map, and $A$ is a positive-semidefinite matrix. For more details on the PLQ family, see \cite{aravkin2013sparse,aravkin2014orthogonal}. For example, the elastic net penalty~\cite{elastic}, used for group detection, and the soft-insensitive loss~\cite{soft}, used for training Support Vector Machines, fall within this class.

Note that the assumptions of dual nondegeneracy and strict complementarity (Assumptions \ref{ass1} and \ref{ass2}) were only used in the proof Theorem~\ref{thm:quad_comp} to guarantee inequality \eqref{eqn:rel_int}. On the other hand, this inequality holds automatically if the subdifferentials $\partial g^{\star}(-A^T\bar y)$ and $\partial f^{\star}(\bar y)$ are polyhedral---a common situation.

\begin{corollary}[Quadratic growth without strict complementarity]{\hfill \\ }
	Consider a convex PLQ function $g\colon\R^n\to\overline{\R}$ and a $C^1$-smooth convex function $f\colon\R^m\to\R$. Suppose that the function $\varphi(x):=f(Ax)+g(x)$ has a nonempty compact set $S$ of minimizers and that either $f$ is strictly convex or $f$ is PLQ.  
	Then the error bound condition holds with some parameters $(\gamma,\nu)$.
\end{corollary}
\begin{proof}
	Since $g$ is PLQ, the subdifferential $\partial g^{\star}$ at any point is polyhedral.
	Similarly, if $f$ is PLQ then $\partial f^{\star}$ is polyhedral at any point, while if $f$ is strictly convex, the subdifferential $\partial f^{\star}(\bar y)$ is a singleton.
	Thus in all cases the inequality \eqref{eqn:rel_int} holds and the proof proceeds as in Theorem~\ref{thm:quad_comp}.
\end{proof}

Firm convexity is preserved under separable sums.

\begin{lemma}[Separable sum]
	Consider a family of functions $f_i\colon\R^{n_i}\to\overline{\R}$ for $i=1,\ldots, m$ with each $f_i$ firmly convex relative to some $v_i\in \R^{n_i}$. Then the separable function $f:\R^{n_1+\ldots+n_m}\to\overline{\R}$ defined by 
	$f(x)=\sum^m_{i=1} f_i(x_i)$ is firmly convex relative to the vector $(v_1,\ldots, v_n)$.
\end{lemma}
\begin{proof}
	The proof is immediate from definitions.
\end{proof}

Moreover, firmly convex functions are preserved by the Moreau envelope. 

\begin{theorem}[Moreau envelope]
	Consider a function $h\colon\R^n\to\overline{\R}$ that is firmly convex relative to a vector $v$. Then the Moreau envelope $h^{t}$ is itself firmly convex relative to $v$.
\end{theorem}
\begin{proof}
	Define the tilted functions $h_v(x):=h(x)-\langle v,x\rangle$ and
	$(h^t)_v(x):=h^t(x)-\langle v,x\rangle$. Observe 
	\begin{align*}
	(h^t)_v(x)&=\min_y \left\{h(y)+\frac{1}{2t}\|y-x\|^2 -\frac{1}{t}\langle t v,x \rangle \right\}\\
	&=\min_y \left\{h(y)-\langle v,y\rangle +\frac{1}{2t}\|y-(x-tv)\|^2-\frac{t}{2}\|v\|^2 \right\}\\
	&=(h_v)^t(x-tv) -\frac{t}{2}\|v\|^2.
	\end{align*}
	Since firm convexity is invariant under translation of the domain, it is now sufficient to show that $(h_v)^t$ is firmly convex relative to the zero vector. 
	To this end, let $S$ be the set of minimizers of $h_v$, or equivalently the set of minimizers of $(h_v)^t$. 
	Since $h$ is firmly convex relative to $v$, for any compact set $\mathcal{Z}\subset\R^n$, there exists a constant $\widehat{L}\geq 0$ so that 
	$$\dist(z,S)\leq  \widehat{L} \|t^{-1}(z-\prox_{t h_v}(z))\|\qquad \textrm{ for all } z\in \mathcal{Z}.$$
	Taking into account the equation 
	$$\nabla (h_v)^t(z)=t^{-1}[z-\prox_{t h_v}(z)],$$
	we deduce
	$\dist(z,S)\leq \widehat{L}\cdot\|\nabla (h_v)^t(z)\|$ for all $z\in \mathcal{Z}$, thereby completing the proof.
\end{proof}

In summary, all typical smooth penalties (e.g. square $l_2$-norm, logistic loss), polyhedral functions (e.g. $l_1$ and $l_{\infty}$-penalties, vapnik, hinge loss, check function, anisotropic total variation penalty), Moreau envelopes of polyhedral functions (e.g. Huber and quantile huber~\cite{aravkin2014orthogonal}), and general affinely composed PLQ penalties (e.g. soft-insensitive loss~\cite{soft}, elastic net~\cite{elastic}) are firmly convex. Another important example is the nuclear norm \cite{so_error}.

\section{Prox-linear algorithm}\label{sec:prox_lin_alg}
We next step away from convex formulations \eqref{eqn:main_problem}, and consider the broad class of nonsmooth and nonconvex optimization problems 
\begin{equation}\label{eqn:target}
\min_x~ \varphi(x):=g(x)+h(c(x)),
\end{equation}
where $g\colon\R^n\to\overline\R$ is a proper closed convex function, $h\colon\R^m\to\R$ is a finite-valued convex function, and $c\colon\R^n\to\R^m$ is
a $C^1$-smooth mapping.  Since such problems are typically nonconvex, we seek a point $x$ that is only {\em first-order stationary}, meaning that the directional derivate of $\varphi$ at $x$ is nonnegative in all directions. The directional derivate of $\varphi$ is exactly the support function of the 
{\em subdifferential set}
$$\partial \varphi(x):=\partial g(x)+\nabla c(x)^T\partial h\big(c(x)\big),$$ 
and hence stationary of $\varphi$ at $x$ simply amounts to the inclusion $0\in \partial\varphi(x)$.

To specify the algorithm we study, define for any points $x,y\in\R^n$ the linearized function 
$$\varphi(x;y):=g(y)+h\big(c(x)+\nabla c(x)(y-x)\big),$$
and for any real $t>0$ consider the quadratic perturbation
$$\varphi_{t}(x;y):=\varphi(x;y)+\frac{1}{2t}\|x-y\|^2.$$
Note that the function $\varphi(x;\cdot)$ is always convex, even though $\varphi$ typically is not convex.
Let $x^t$ be the minimizer of the proximal subproblem 
$$x^t:=\argmin_y~ \varphi_{t}(x;y).$$

Suppose now that $h$ is $L$-Lipschitz continuous and the Jacobian $\nabla c(x)$ is $\beta$-Lipschitz continuous. 
It is then immediate that the linearized function $\varphi(x;\cdot)$ is quadratically close to $\varphi$ itself:
\begin{equation}\label{eqn:upper-lower}
-\frac{L\beta}{2}\|x-y\|^2\leq \varphi(y)-\varphi(x;y)\leq  \frac{L\beta}{2}\|x-y\|^2.
\end{equation}
In particular, 
$\varphi_t(x;\cdot)$ is a quadratic upper estimator of $\varphi$ for any $t \leq (L\beta)^{-1}$. 
We now define the prox-gradient mapping in the natural way
$$\mathcal{G}_t(x):=t^{-1}(x-x^t).$$
It is easily verified that equality $\mathcal{G}_t(x)=0$ holds if and only if $x$ is stationary for $\varphi$. In this section, we consider the well-known prox-linear method (Algorithm \ref{algo:prox_lin_simple}), recently studied for example in \cite{prx_lin}; see also \cite{composite_cart} for interesting variants. The ideas behind the method (and its trust-region versions) go back a long time, e.g. \cite{powell_paper,burke_com,yu_super,steph_conv_comp,fletcher_back,pow_glob,gn_burke,burke_second_order,burke_poli}; see \cite{burke_com} for a historical discussion. We note that Algorithm~\ref{algo:prox_lin_simple} differs slightly from the one in \cite{prx_lin} in the step acceptance criterion.



%
%
%

\begin{algorithm}[h!]
	\DontPrintSemicolon 
	\KwData{A point $x_1\in \dom g$, and constants $q \in (0,1)$, $t>0$, and $\sigma> 0$.
	}
	
	$k\gets 1$\;
	
	\While{$\|\mathcal{G}_t(x_k)\| > \epsilon$}{
		\While{$\varphi(x^t_k)> \varphi(x_k)-\frac{\sigma}{2}\|\mathcal{G}_t(x_k)\|^2$ }{
			$t \gets qt$		 		\Comment*{Backtracking}
		}
		$x_{k+1}\gets x_k^t$
		\Comment*{Iterate update}
		$k\gets k+1$\;
	}
	\Return $x_{k}$\;
	\caption{Prox-linear method}
	\label{algo:prox_lin_simple}
\end{algorithm}

%

Note that the prox-gradient method in Section~\ref{sec:lin_con} for the problem  $\min_x f(x)+g(x)$ is an example of Algorithm~\ref{algo:prox_lin_simple} with the decomposition $c(x)=f(x)$ and $h(r)=r$. 
In this case, we have $L=1$. Observe also that we do not require $f$ to be convex anymore. Motivated by this observation, we now perform an analysis following the same strategy as for the proximal gradient method; there are important and surprising differences, however, both in the conclusions we make and in the proof techniques.
We begin with the following lemma; the proof follows that of \cite[Lemma~2.3.2]{Nes04b}. 
\begin{lemma}[Gradient inequality]
	For all points $x,y\in \R^n$, the inequality
	\begin{equation}\label{eqn:main_ineq2}
	\varphi(x;y)\geq \varphi_t(x;x^t)+\langle \mathcal{G}_t(x), y-x\rangle+\frac{t}{2}\|\mathcal{G}_t(x)\|^2,
	\end{equation}
	holds, and consequently we have 
	\begin{equation}\label{eqn:funky}
	\varphi(y)\geq \varphi(x^t)+\langle \mathcal{G}_t(x), y-x\rangle+\frac{t}{2}(2-L\beta t)\|\mathcal{G}_t(x)\|^2 - \frac{L\beta}{2}\|x-y\|^2.
	\end{equation}
	and
	\begin{equation}\label{eqn:descent2}
	\varphi(x)\geq \varphi(x^t)+\frac{t}{2}(2-L\beta t)\|\mathcal{G}_t(x)\|^2.
	\end{equation}
\end{lemma}
\begin{proof}
	Noting that the function  
	$\varphi_t(x;y):=\varphi(x;y)+\frac{1}{2t}\|y-x\|^2$ is strongly convex in the variable $y$, we deduce
	\begin{align*}
	\varphi(x;y)&= \varphi_{t}(x;y)-\frac{1}{2t}\|y-x\|^2\\
	&\geq \varphi_t(x;x^t)+\frac{1}{2t}\left(\|y-x^t\|^2- \|y-x\|^2\right)\\
	&= \varphi_t(x;x^t)+\langle \mathcal{G}_t(x), y-x\rangle+\frac{t}{2}\|\mathcal{G}_t(x)\|^2,
	\end{align*}
	establishing \eqref{eqn:main_ineq2}. Inequality \eqref{eqn:funky} follows by combining  \eqref{eqn:upper-lower} and \eqref{eqn:main_ineq2}.
	Finally, we obtain inequality \eqref{eqn:descent2} from 
	\eqref{eqn:funky} by setting $y=x$.
\end{proof}

For simplicity, we assume that the constants $L$ and $\beta$ are known and we set $t\leq\frac{1}{L\beta}$ and $\sigma=\frac{1}{L\beta}$ in Algorithm~\ref{algo:prox_lin_simple}, so that the line search always accepts the initial step. The more general setting with the backtracking line-search is entirely analogous. Observe now that the inequality $\eqref{eqn:descent2}$  yields the functional decrease guarantee  
\begin{equation}\label{eqn:decr_cond}
\varphi(x_{k+1})\leq \varphi(x_k)-\frac{1}{2L\beta}\|\mathcal{G}_t(x_k)\|^2,\end{equation}
and hence we obtain the global convergence rate
$$\min_{i=1,\ldots,k} \|\mathcal{G}_t(x_i)\|^2\leq \frac{2L\beta}{k}\sum^k_{i=1} \varphi(x_i)-\varphi(x_{i+1})=\frac{2L\beta (\varphi(x_1)- \varphi^*)}{k},$$
where $\varphi^*:=\lim_{k\to \infty} \varphi(x_k)$. Note moreover the prox-gradients $\mathcal{G}_t(x_k)$ tend to zero, since their norms are square-summable. Thus after $2L\beta(\varphi(x_1)- \varphi^*)/\epsilon$ iterations, we can be sure that Algorithm~\ref{algo:prox_lin_simple} finds a point $x_k$ satisfying $\|\mathcal{G}_t(x_k)\|^2\leq \epsilon$. 

Does a small stepsize $\|\mathcal{G}_t(x)\|$ imply that $x$ is ``nearly stationary''? 
This question is fundamental, and speaks directly to reliability of the termination criterion of Algorithm \ref{algo:prox_lin_simple}. We will see shortly (Theorem~\ref{thm:perturb}) that the answer is affirmative: if the quantity $\|\mathcal{G}_t(x)\|$ is small, then $x$ is close to a point that is nearly stationary for $\varphi$. Our key tool for establishing this result will be Ekeland's variational principle.

\begin{theorem}[Ekeland's variational principle] {\hfill \\ }
	Consider a closed function $f\colon\R^n\to\overline{\R}$ that is bounded from below.	Suppose that for some $\epsilon>0$ and $\bar x\in \R^n$, we have $f(\bar x)\leq \inf f+\epsilon$. Then for any $\rho >0$, there exists a point $\bar u$ satisfying
	\begin{itemize}
		\item $f(\bar u)\leq f(\bar x)$,
		\item $\|\bar x-\bar u\|\leq \epsilon/\rho$,
		\item $\displaystyle\{\bar u\}=\argmin_u\, \{f(u)+\rho\|u-\bar u\|\}$. 
	\end{itemize}	
\end{theorem}

We can now explain the relation of the quantity $\|\mathcal{G}_t(x)\|$ to approximate stationarity.
Aside from its immediate appeal, this result will play a central role both in the proof of Theorem~\ref{thm:main_prop} and in \cref{sec:conv_rate_tilt}.

\begin{theorem}[Prox-gradient and near-stationarity]\label{thm:perturb} {\hfill \\ }
	Consider the convex-composite problem \eqref{eqn:target}, where $h$ is $L$-Lipschitz continuous and the Jacobian $\nabla c$ is $\beta$-Lipschitz.
	Then for any real $t>0$, there exists a point $\hat x$ satisfying the properties
	\begin{enumerate}[(i)]
		\item\label{it:proximity} (point proximity) $\quad\|x^t-\hat x\|\leq \|x^t-x\|$,
		\item \label{it:near_func} (value proximity)
		$\quad\varphi(\hat x)-\varphi(x^t)\leq \frac{t}{2}(L\beta t+1)\|\mathcal{G}_t(x)\|^2$, 
		\item\label{it:near_stat} (near-stationarity) $\quad\dist\left(0;\partial \varphi(\hat x)\right)\leq (3L\beta t +2)\|\mathcal{G}_t(x)\|$.
	\end{enumerate}
\end{theorem}
\begin{proof}
	Define the function $\zeta(y):=\varphi(y)+\frac{1}{2}(L\beta+t^{-1})\|x-y\|^2$ and note that the inequality $\zeta(y)\geq \varphi_t(x,y)\geq \varphi_t(x,x^t)$ holds for all $y\in \R^n$. Letting $\zeta^*$ be the infimal value of $\zeta$, we have
	$$\zeta(x^t)-\zeta^*\leq \varphi(x^t)-\varphi(x;x^t) +\frac{L\beta}{2}\|x^t-x\|^2\leq L\beta\|x^t-x\|^2,$$ 
	where the last inequality follows from \eqref{eqn:upper-lower}.
	Define the constants $\epsilon:=L\beta\|x^t-x\|^2$ and $\rho:= L\beta\|x^t-x\|$.
	Applying Ekeland's variational principle, we obtain
	a point $\hat x$ satisfying the inequalities 
	$\zeta(\hat x)\leq \zeta(x^t)$ and $\|x^t-\hat x\|\leq \epsilon/\rho$, and the inclusion $0\in \partial \zeta(\hat x)+\rho {\bf B}$. The proximity conditions $(\ref{it:proximity})$ and $(\ref{it:near_func})$ are immediate. 
	To see near-stationarity $(\ref{it:near_stat})$, observe
	\begin{align*} 
	\dist(0,\partial \varphi(\hat x))&\leq \rho+(L\beta+t^{-1})\|\hat x -x\|\\
	&\leq L\beta\|x^t-x\|+(L\beta+t^{-1})(\|x^t-\hat x\|+\|x^t-x\|)\\
	&\leq \big(L\beta +2(L\beta+t^{-1})\big)\|x^t-x\|.
	\end{align*}
	The result follows.
\end{proof}

Following the general outline of the paper and armed with Theorem~\ref{thm:perturb}, we now turn to linear convergence.
To this end, let $\{x_k\}$ be a sequence generated by Algorithm~\ref{algo:prox_lin_simple} and suppose that $x^*$ is a limit point of $\{x_k\}$. Then inequality \eqref{eqn:funky} (with $y=x^*$) and lower-semicontinuity of $\varphi$ immediately imply that the values $\varphi(x_k)$ converge to $\varphi(x^*)$. A standard argument also shows that $x^*$ is a stationary point of $\varphi$.
Appealing to inequality  \eqref{eqn:funky} with $y=x^*$, we deduce
$$\varphi(x_{k+1})-\varphi(x^*)\leq \|\mathcal{G}_t(x_k)\|^2\left(\frac{\|x_k-x^*\|}{\|\mathcal{G}_t(x_k)\|}  +\frac{L\beta}{2}\frac{\|x_k-x^*\|^2}{\|\mathcal{G}_t(x_k)\|^2}+\frac{t}{2}(L\beta t-2)\right).$$
Defining $\gamma_k:=\max\{1,\|x_k-x^*\|/\|\mathcal{G}_t(x_k)\|\}$  we deduce
$$\varphi(x_{k+1})-\varphi(x^*)\leq \|\mathcal{G}_t(x_k)\|^2\left(\left(1+\frac{L\beta}{2}\right)\gamma_k^2+\frac{t}{2}(L\beta t-2)\right).$$
Combining this inequality with \eqref{eqn:decr_cond} yields the geometric decay
$$\varphi(x_{k+1})-\varphi(x^*)\leq \left(1-\frac{1}{L\beta(2+L\beta)\gamma^2_k} \right)(\varphi(x_{k})-\varphi(x^*)).$$
Hence provided that $\gamma_k$ are bounded for all large $k$, the function values asymptotically converge Q-linearly. This motivates the following definition, akin to Definition~\ref{defn:error_bound1}.

\begin{definition}[Error bound condition]\label{defn:erro_gen}
	{\rm
		We say that the {\em error bound condition holds around a point $\bar x$ with parameter $\gamma>0$} if there exists a real number $\epsilon >0$ so that the inequality
		$$\dist\left(x,(\partial \varphi\right)^{-1}(0))\leq \gamma \|\mathcal{G}_t(x)\|\quad \textrm{ is valid for all }\quad x\in B_{\epsilon}(\bar{x}).$$
	}
\end{definition}

Hence we arrive at the following convergence guarantee.
\begin{theorem}[Linear convergence of the proximal method]\label{thm:lin_conv_bad}
	Consider the sequence $x_k$ generated by Algorithm~\ref{algo:prox_lin_simple} with $t\leq (L\beta)^{-1}$, and suppose that $x_k$ has some limit point $x^*$ around which the error bound condition holds with parameter $\gamma>0$. Define now the fraction $$q:=1-\frac{1}{L\beta(2+L\beta)\gamma^2}.$$
	Then for all large $k$, function values converge Q-linearly 
	$$\varphi(x_{k+1})-\varphi(x^*)\leq q\cdot(\varphi(x_{k})-\varphi(x^*)),$$
	while the points $x_k$ asymptotically converge R-linearly:
	there exists an index $r$ such that the inequality
	$$\|x_{r+k}-x^*\|^2\leq C\cdot (\varphi(x_r)-\varphi^*)\cdot q^{k},$$
	holds for all $k\geq 1$, where we set $C:=\frac{2}{{L\beta(1-\sqrt{1-(1-q)^{-1}})^2}}$.
\end{theorem}
\begin{proof}
	Let $\epsilon,\nu>0$ be as in Definition~\ref{defn:erro_gen}. As observed above, the function values $\varphi(x_k)$ converge to $\varphi(x^*)$. Hence we may assume all the iterates $x_k$ lie in $[\varphi\leq \varphi(x^*)+\nu]$.
	We aim now to show that if $x_r$ is sufficiently close to $x^*$, then all following iterates never leave the ball $B_{\epsilon}(x^*)$. To this end, let $r$ be an index such that $x_r$ lies in $B_{\epsilon}(x^*)$ and let $k\geq 1$ be the smallest index satisfying $x_{r+k}\notin B_{\epsilon}(x^*)$. Defining $\zeta:=L\beta(2+L\beta)\gamma^2$ and using inequality \eqref{eqn:decr_cond}, 
	we deduce 
	\begin{align*}
	\|x_k-x_r\|&\leq \sum^{k-1}_{i=r}\|x_{i}-x_{i+1}\|\leq \sqrt{2/(L\beta)}\sum^{k-1}_{i=r}\sqrt{\varphi(x_i)-\varphi(x_{i+1})} \\
	&\leq\sqrt{2/(L\beta)}\sqrt{\varphi(x_r)-\varphi(x^*)}\sum^{k}_{i=r} \left(1-\frac{1}{\zeta}\right)^{\frac{i-r}{2}}\leq \sqrt{\frac{2(\varphi(x_r)-\varphi(x^*))}{L\beta(1-\zeta^{-1})}}.
	\end{align*}
	Hence if $x_r$ lies in the ball $B_{\epsilon/2}(x^*)$ and is sufficiently close to $x^*$ so that the right-hand-side is smaller than $\frac{\epsilon}{2}$, we obtain a contradiction. Thus there exists an index $r$ so that for all $k\geq r$, the iterates $x_k$ lie in $B_{\epsilon}(x^*)$. The claimed Q-Linear rate follows immediately.
	To obtain the R-linear rate of the iterates, we argue as in the proof of Theorem~\ref{thm:lin_conv}:
	\begin{align*}
	\|x_{r+k}&-x^*\|\leq \sum^{\infty}_{i=r+k}\|x_{i}-x_{i+1}\|\leq \sqrt{2/(L\beta)}\sum^{\infty}_{i=r+k}\sqrt{\varphi(x_i)-\varphi(x_{i+1})} \\
	&\leq\sqrt{2/(L\beta)}\sqrt{\varphi(x_r)-\varphi(x^*)}\sum^{\infty}_{i=r+k} \left(1-\frac{1}{\zeta}\right)^{\frac{i-r}{2}}\leq
	\left(1-\frac{1}{\zeta}\right)^{k/2}D\sqrt{\varphi(x_r)-\varphi^*},
	\end{align*}
	where $D=\frac{\sqrt{2}}{{\sqrt{L\beta}(1-\sqrt{1-\zeta^{-1}})}}$. Squaring both sides, the result follows.
\end{proof}

\Cref{thm:lin_conv_bad} already marks a point of departure from the convex setting. Gradient descent for an $\alpha$-strongly convex function with $\beta$-Lipschitz gradient converges at the linear rate $1-\alpha/\beta$. Treating this setting as a special case of composite minimization with $g=0$ and $h(t)=t$, Theorem~\ref{thm:lin_conv_bad} guarantees the linear rate only on the order of $1-(\alpha/\beta)^2$.
The difference is the lack of convexity;  the linearizations $\varphi(x;\cdot)$ no longer lower bound the objective function $\varphi$, but only do so up to a quadratic deviation, thereby leading to a worse linear rate of convergence. We put this issue aside for the moment and will revisit it in \cref{sec:conv_rate_tilt}, where we will show that Algorithm~\ref{algo:prox_lin_simple} accelerates under a natural {\em uniform} quadratic growth condition. The ensuing discussion relating the error bound condition and quadratic growth will drive that analysis as well.

Following the pattern of the current work, we seek to interpret the error bound condition in terms of a natural property of the subdifferential $\partial \varphi$. It is tempting to proceed by showing that the step-lengths of Algorithm~\ref{algo:prox_lin_simple} are proportional to the step-lengths of the proximal point method $z_{k+1}\in (I+t\partial \varphi)^{-1}(z_k)$, as in Theorem~\ref{thm:approx}. The following theorem attempts to do just that. First, we establish inequality \eqref{eqn:ez}, showing that the norm $\|\mathcal{G}_t(x)\|$ is always bounded by twice the stationarity measure  
$\dist\left(0; \partial \varphi(x)\right)$. This is a direct analogue of \eqref{eqn:easy}, though we arrive at it through a different argument. Second, seeking to imitate the key inequality \eqref{eqn:main_ineq}, we arrive at the inequality \eqref{eqn:not_needed} below. The difficulty is 
that in this more general setting, the proportionality constant we need (left-hand-side of  \eqref{eqn:not_needed}) tends to $+\infty$ as $t$ tends to $(L\beta)^{-1}$ -- the most interesting regime. 
We will circumvent this difficulty in the proof of our main result (Theorem~\ref{thm:main_prop}) by a separate argument using Theorem~\ref{thm:perturb}; nonetheless, we believe the proportionality inequality \eqref{eqn:not_needed} is interesting in its own right.



\begin{theorem}[Step-lengths comparison]\label{thm:approx2} {\hfill \\ } 
	Consider the convex-composite problem \eqref{eqn:target}.
	Then the inequality 
	\begin{equation}\label{eqn:ez}	
	\frac{1}{2}\|\mathcal{G}_t(x)\|\leq \dist\left(0; \partial \varphi(x)\right)\qquad \textrm{ holds.}
	\end{equation} 
	Suppose in addition that $h$ is L-Lipschitz continuous and $\nabla c$ is $\beta$-Lipschitz continuous, and set $r:=L\beta$. Then for $t<r^{-1}$ and any point $x^+\in (I+t\partial \varphi)^{-1}(x)$ the inequality holds:
	\begin{equation}\label{eqn:not_needed}
	1+(1-tr)^{-1}\geq \frac{\|x^t-x\|}{\|x^+-x\|}+\frac{\|x^+-x\|}{\|x^t-x\|}.
	\end{equation}
\end{theorem}
\begin{proof}
	Fix a vector $v\in \partial \varphi(x)$. Then there exist vectors $z\in \partial g(x)$ and $w\in \partial h(c(x))$ satisfying $v=z+\nabla c(x)^*w$.
	Convexity yields 
	\begin{align*}
	g(x^t)+&h\big(c(x)+\nabla c(x)(x^t-x)\big)\geq\\
	&\big(g(x)+\langle z,x^t-x\rangle\big)+\big( h\big(c(x)\big)+\langle w,\nabla c(x)(x^t-x) \rangle\big)
	=\varphi(x)+\langle v, x^t-x \rangle.
	\end{align*}
	Appealing to the inequality $\varphi(x)\geq \varphi(x;x^t)+\frac{1}{2t}\|x-x^t\|^2$, we deduce
	$\|v\|\cdot\|x^t-x\|\geq \frac{1}{2t}\|x^t-x\|^2$, completing the proof of \eqref{eqn:ez}.	
	
	Next, again fix a point $x$ and a subgradient $v=z+\nabla c(x)^*w$ for some $z\in \partial g(x)$ and $w\in \partial h(c(x))$. Then for all $y\in \R^n$ we successively deduce 
	\begin{align*}
	\varphi(y)=g(y)+h(c(y))&\geq \big( g(x)+\langle z,y-x\rangle\big) +\big(h(c(x))+\langle w,c(y)-c(x)\rangle\big)\\
	&\geq \varphi(x)+\langle z,y-x\rangle+\langle w,\nabla c(x)(y-x)\rangle -\frac{L\beta}{2}\|y-x\|^2  \\
	&=\varphi(x)+\langle v,y-x\rangle-\frac{L\beta}{2}\|y-x\|^2.
	\end{align*}
	Consider now a point $x^+\in (I+t\partial \varphi)^{-1}(x)$.
	Setting $r:=L\beta$ and replacing $x$ with $x^+$ in the above inequality, we deduce 
	\begin{align*}
	\varphi(y)&\geq \varphi(x^+)+\langle t^{-1}(x-x^+),y-x^+\rangle-\frac{L\beta}{2}\|y-x^+\|^2\\
	&=\varphi(x^+)+\frac{1}{2t}\|x^+-x\|^2+\frac{t^{-1}-r}{2}\|x^+-y\|^2-\frac{1}{2t}\|y-x\|^2
	\end{align*}
	Plugging in $y=x^t$, we deduce
	\begin{align*}
	\frac{t^{-1}-r}{2}\|x^+-x^t\|^2&\leq \big(\varphi(x^t)+\frac{1}{2t}\|x^t-x\|^2\big)- \big(\varphi(x^+)+\frac{1}{2t}\|x^+-x\|^2\big)\\
	&\leq \varphi_t(x;x^t)- \varphi_t(x;x^+) +\frac{r}{2}\left(\|x^t-x\|^2+\|x^+-x\|^2\right).
	\end{align*}
	Taking into account the strong convexity inequality
	$\varphi_t(x;x^+)\geq \varphi_t(x;x^t)+\frac{1}{2t}\|x^+-x^t\|^2$, we deduce
	$$(2t^{-1}-r)\|x^+-x^t\|^2\leq r\left(\|x^t-x\|^2+\|x^+-x\|^2\right).$$
	A short computation then shows
	$$1+(1-tr)^{-1}\geq \frac{\|x^t-x\|}{\|x^+-x\|}+\frac{\|x^+-x\|}{\|x^t-x\|},$$ 
	and the result follows.
\end{proof}

As alluded to prior to the theorem, the inequality \eqref{eqn:not_needed} is meaningless for
$t= 1/L\beta$ -- the most interesting case -- since the left-hand-side becomes infinite. This seems unavoidable.  In essence, the difficulty is that the base-point $x$ at which the lengths comparison is made remains fixed. To circumvent this difficulty, instead of relying on \eqref{eqn:not_needed}, we will  prove our main result (Theorem~\ref{thm:main_prop}) by using Theorem~\ref{thm:perturb}. 


Next, we introduce the ``natural property'' of the subdifferential with which we will equate the error bound.



\begin{definition}[Subregularity]\label{defn:subreg}
	{\rm A set-valued mapping $F\colon\R^n\rightrightarrows \R^m$ is {\em subregular} at $(\bar{x},\bar{y})\in \gph F$ with constant $l>0$ if there exists a neighborhood $\mathcal{X}$ of $\bar{x}$ satisfying 
		$$\dist\left(x; F^{-1}(\bar{y})\right)\leq l\cdot\dist\left(\bar y; F(x)\right)\qquad\textrm{for all }x\in \mathcal X.$$}
\end{definition}

Clearly, the error bound property around a stationary point $\bar x$ of $\varphi$ with parameter $\gamma$ amounts to subregularity of the prox-gradient mapping $\mathcal{G}_t(\cdot)$ at $(\bar{x},0)$ with constant $\gamma$.
We aim to show that the error bound property is equivalent to subregularity of the subdifferential $\partial \varphi$ itself -- a transparent notion closely tied to quadratic growth \cite{tilt_other, crit_semi, tilt}. We first record the following elementary lemma; we omit the proof, as it quickly follows from definitions.

\begin{lemma}[Perturbation by identity]\label{lem:perturb_ident}
	Consider a set-valued mapping $S\colon\R^n\rightrightarrows\R^m$ and a pair $(\bar x,0)\in \gph S$.
	Then if $S$ is subregular at $(\bar x,0)$ with constant $l$, the mapping $(I+S^{-1})^{-1}$ is subregular at $(\bar x,0)$ with constant $1+l$.	Conversely, if $(I+S^{-1})^{-1}$ is subregular at $(\bar x,0)$ with constant $\hat{l}$, then $S$ is subregular at $(\bar x,0)$ with constant $1+\hat{l}$.
\end{lemma}

The following result analogous to Theorem~\ref{thm:prox_sub} is now immediate.

\begin{theorem}[Proximal and subdifferential subregularity]\label{thm:prox_sub3}{\hfill \\ }
	Consider the convex-composite problem \eqref{eqn:target} and let $\bar{x}$ be a stationary point of $\varphi$. Consider the conditions
	\begin{enumerate}[(i)]
		\item\label{eqn:subreg_sub2} the subdifferential $\partial \varphi$ is subregular at $(\bar x,0)$ with constant $l$.
		\item\label{eqn:subreg_prox2} the mapping $T:=t^{-1}\left(I-(I+t\partial \varphi)^{-1}\right)$ is subregular at $(\bar{x}, 0)$ with constant $\hat{l}$.
	\end{enumerate}
	If condition $(\ref{eqn:subreg_sub2})$ holds, then so does condition $(\ref{eqn:subreg_prox2})$ with $\hat{l}=l+t$. Conversely, condition $(\ref{eqn:subreg_prox2})$ implies condition $(\ref{eqn:subreg_sub2})$ with $l=\hat{l}+t$.
\end{theorem}
\begin{proof}
	Note first the equality  $tT=(I+(t\partial \varphi)^{-1})^{-1}$ (equation~\eqref{prox_formula}).
	Suppose that $\partial \varphi$ is $l$-subregular  at $(\bar{x},0)$ with constant $l$.  
	Then by Lemma~\ref{lem:perturb_ident}, the mapping $tT$
	is subregular at $(\bar x,0)$ with constant $1+l/t$, and hence $T$ is subregular at $(\bar x,0)$ with constant $l+t$, as claimed.
	The converse argument is  analogous.
\end{proof}

We are now ready to prove the main result of this section.
\begin{theorem}[Subdifferential subregularity and the error bound property] \label{thm:main_prop}{\hfill \\ }
	Consider the convex-composite problem \eqref{eqn:target}, where $h$ is $L$-Lipschitz continuous and the Jacobian $\nabla c$ is $\beta$-Lipschitz. Let $\bar{x}$ be a stationary point of $\varphi$ and consider the conditions:
	\begin{enumerate}[(i)]
		\item\label{eqn:subreg_sub3} the subdifferential $\partial \varphi$ is subregular at $(\bar x,0)$ with constant $l$.
		\item\label{eqn:subreg_prox3} the prox-gradient mapping $\mathcal{G}_t(\cdot)$ is subregular at $(\bar x,0)$ with constant $\hat l$.
	\end{enumerate}
	If condition $(\ref{eqn:subreg_prox3})$ holds, then condition~$(\ref{eqn:subreg_sub3})$ holds with $l:=2\hat l$. Conversely,
	if condition $(\ref{eqn:subreg_sub3})$ holds, then condition $(\ref{eqn:subreg_prox3})$ holds with $\hat l:=(3L\beta t  +2)l+2t$.
\end{theorem}
\begin{proof}
	Suppose first that the gradient mapping $\mathcal{G}_t(x)$ is subregular at $(\bar x,0)$ with constant $\hat l$. Then  by Theorems~\ref{thm:approx2} and \ref{thm:prox_sub3}, we deduce  for all $x$ near $\bar{x}$, the inequalities
	$$\dist(x,\left(\partial \varphi)^{-1}(0)\right)\leq \hat l\cdot \|\mathcal{G}_t(x)\|\leq 2 \hat l\cdot \dist(0;\partial \varphi(x)).$$
	Conversely, suppose that $\partial \varphi$ is subregular at $(\bar x,0)$ with constant $l$. Fix a point $x$, and let $\hat x$ be the point guaranteed to exist by Theorem~\ref{thm:perturb}.
	For the purpose of establishing subregularity of $\mathcal{G}_t(\cdot)$, we can suppose the $x$ and $x_t$ are arbitrarily close to $\bar x$.
	Then $\hat x$ is close to $\bar x$, and we deduce 
	\begin{align*}
	l\cdot \dist(0,\partial \varphi(\hat x))\geq \dist\left(\hat x; (\partial \varphi)^{-1}(0)\right)&\geq \dist\left(x; (\partial \varphi)^{-1}(0)\right)-\|x^t-\hat x\|-\|x^t-x\|\\
	&\geq  \dist\left(x; (\partial \varphi)^{-1}(0)\right)-2\|x^t-x\|.
	\end{align*}
	We conclude $\dist\left(x; (\partial \varphi)^{-1}(0)\right)\leq \left( l(3L\beta  +2t^{-1})+2\right)\|x^t-x\|$, as claimed.
\end{proof}

Thus subregularity of the subdifferential $\partial \varphi$ and the error bound property are identical notions, with a precise relationship between the constants. Subdifferential subregularity at a minimizer, on the other hand, is equivalent to the natural quadratic growth condition when the functions in question are semi-algebraic (or more generally tame) \cite{crit_semi} or convex (Theorem~\ref{thm:subdif_quad}).
To the best of our knowledge, it is not yet known if such a relationship persists for all convex-composite functions.

\section{Natural rate of convergence under tilt-stability}\label{sec:conv_rate_tilt}
As we alluded to in \cref{sec:prox_lin_alg}, the linear rate at which Algorithm~\ref{algo:prox_lin_simple} converges under the error bound condition is an order of magnitude slower than the rate that one would expect. In \cref{sec:prox_lin_alg}, we highlighted the equivalence between the error bound condition and subregularity of the subdifferential, and their close relationships to quadratic growth. We will now show that when these properties hold {\em uniformly} relative to tilt-perturbations, the algorithm accelerates to the natural rate. 

We begin with the following definition.
\begin{definition}[Stable strong local minimizer]
	{\rm We say that $\bar x$ is a {\em stable strong local minimizer with constant $\alpha>0$} of a function $f\colon\R^n\to\overline{\R}$ if there exists a neighborhood $\mathcal{X}$ of $\bar x$ so that for each vector $v$ near the origin, there is a point $x_v$ (necessarily unique) in $\mathcal{X}$, with $x_0=\bar x$, so that in terms of the perturbed functions $f_v:=f(\cdot)-\langle v,\cdot\rangle$, the inequality
		$$f_v(x)\geq f_v(x_v)+\frac{\alpha}{2}\|x-x_v\|^2 \qquad \textrm{holds for all } x\in \mathcal{X}.$$}
\end{definition}

This type of uniform quadratic growth is known to be equivalent to a number of influential notions, such as tilt-stability \cite{tilt_orig} and strong metric regularity of the subdifferential \cite{tilt}. Here, we specialize the discussion to the convex-composite case, though the relationships hold much more generally. The following theorem appears in \cite[Theorem~3.7, Proposition~4.5]{tilt_other}; some predecessors were proved in \cite{tilt,tilt_inf}.
\begin{theorem}[Uniform growth, tilt-stability, \& strong regularity]\label{thm:eqv_many}{\hfill \\ }
	Consider the convex-composite problem \eqref{eqn:target} and let $x^*$ be a local minimizer of $\varphi$. Then the following properties are equivalent.
	\begin{enumerate}
		\item {\bf (uniform quadratic growth)} The point $x^*$ is a stable strong local minimizer of $\varphi$ with constant $\alpha$.
		\item {\bf (local subdifferential convexity)} There exists a neighborhood $\mathcal{X}$ of $x^*$ so that for any sufficiently small vector $v$, there is a point $x_v\in \mathcal{X}\cap (\partial \varphi)^{-1}(v)$ so that the inequality 
		$$\varphi(x)\geq \varphi(x_v)+\langle v,x-x_v\rangle+\frac{\alpha}{2}\|x-x_v\|^2 \qquad \textrm{ holds for all }x\in \mathcal{X}.$$
		\item {\bf (tilt-stability)} There exists a neighborhood $\mathcal{X}$ of $ x^*$ so that the mapping 
		$$v\mapsto \argmin_{x\in \mathcal{X}}\{\varphi(x)-\langle v,x\rangle\}$$
		is single-valued and $1/\alpha$-Lipschitz continuous on some neighborhood of the origin.
		\item {\bf (strong regularity of the subdifferential)}\label{it:strong_reg} There exist neighborhoods $\mathcal{X}$ of $x^*$ and $\mathcal{V}$ of $\bar v=0$ so that the restriction $(\partial \varphi)^{-1}\colon \mathcal{V}\rightrightarrows \mathcal{X}$ is a single-valued $1/\alpha$-Lipschitz continuous mapping.
	\end{enumerate}
	
\end{theorem}

There has been a lot of recent work aimed at characterizing the above properties in concrete circumstances; see e.g. \cite{nonlin_prog_tilt,titl_out,ramirez,tilt_with_rock}.
Suppose that the equivalent conditions in Theorem~\ref{thm:eqv_many} hold. We will now investigate the impact of such an assumption on the linear convergence of Algorithm~\ref{algo:prox_lin_simple}. Assume $t\leq (L\beta)^{-1}$. Observe that property \ref{it:strong_reg} directly implies that $\partial \varphi$ is subregular at $(x^*,0)$ with constant $1/\alpha$. Consequently, local linear convergence with the rate on the order of $1-(\frac{\alpha}{\beta L})^2$ is already assured by Theorems \ref{thm:lin_conv_bad} and \ref{thm:main_prop}; our goal is to derive a faster rate.

Consider a point $x$ near $\bar x$. Let $\hat x$ then be the point guaranteed to exist by Theorem~\ref{thm:perturb}, and set $v$ to be the minimal norm vector in the subdifferential $\partial \varphi(\hat x)$. Note the inequality $\|v\|\leq (3L\beta t +2)\|\mathcal{G}_t(x)\|\leq 5\|\mathcal{G}_t(x)\|$. Hence 
for any point $y$ near $x^*$, we have   
\begin{align*}
\varphi(y)&\geq \varphi(\hat x)+\langle v,y-\hat x\rangle &\qquad \textrm{(strong-regularity)}\\
&\geq \varphi_t(x,\hat x)-L\beta \|x-\hat x\|^2+\langle v,y-\hat x\rangle &\qquad \textrm{(inequality~\eqref{eqn:upper-lower})}\\
&\geq \varphi_t(x,x^t)-L\beta \|x-\hat x\|^2+\langle v,y-\hat x\rangle &\qquad \textrm{(definition of $x^t$)}\\
&\geq \varphi(x^t)-L\beta \|x-\hat x\|^2+\langle v,y-\hat x\rangle&\qquad\textrm{ (inequality~\eqref{eqn:upper-lower})}.
\end{align*}
Plugging in $y=x^*$, we deduce 
\begin{align*}
\varphi(x^t)-\varphi(x^*)&\leq L\beta \|x-\hat x\|^2+\langle v,\hat x-x^*\rangle\\
&\leq 4L\beta\|x^t-x\|^2 +\|v\|\cdot(\|\hat x- x^t\|+\|x^t-x\|+\|x-x^*\|)\\
&\leq \frac{4}{L\beta} \|\mathcal{G}_t(x)\|^2 +5\|\mathcal{G}_t(x)\|\cdot(2\|x^t-x\|+\|x-x^*\|)\\
&\leq \frac{14}{L\beta}\|\mathcal{G}_t(x)\|^2 +5\|\mathcal{G}_t(x)\|\cdot \|x-x^*\|\\
&=  \frac{\|\mathcal{G}_t(x)\|^{2}}{L\beta}\left(14+5L\beta\frac{\|x-x^*\|}{\|\mathcal{G}_t(x)\|}\right).
\end{align*}
Hence if while Algorithm~\ref{algo:prox_lin_simple} is running, the fractions $\frac{\|x_k-x^*\|}{\|\mathcal{G}_t(x_k)\|}$ remain bounded by a constant $\gamma$, appealing to the descent inequality~\eqref{eqn:decr_cond}, we obtain the Q-linear convergence  guarantee
\begin{equation}\label{eqn:nat_conv}
\varphi(x_{k+1})-\varphi(x^*)\leq \left( 1-\frac{1}{25+10L\beta\gamma}\right)(\varphi(x_k)-\varphi(x^*)).
\end{equation}
We have thus established the main result of this section.

\begin{theorem}[Natural rate of convergence]\label{thm:nat_rate}
	Consider the convex-composite problem \eqref{eqn:target}, where $h$ is $L$-Lipschitz continuous and the Jacobian $\nabla c$ is $\beta$-Lipschitz. Let $x_k$ be the sequence generated by Algorithm~\ref{algo:prox_lin_simple} with $t \leq (L\beta)^{-1}$. Suppose that $x_k$ has some limit point $x^*$ around which one of the equivalent properties in 
	Theorem~\ref{thm:eqv_many} hold. Define the fraction $$q:=1-\frac{1}{45+50L\beta/\alpha}.$$
	Then for all large $k$, function values converge Q-linearly
	$$\varphi(x_{k+1})-\varphi(x^*)\leq q\cdot(\varphi(x_k)-\varphi(x^*)),$$
	while the points $x_k$ asymptotically converge R-linearly: there exists an index $r$ such that the inequality
	$$\|x_{r+k}-x^*\|^2\leq q^{k}\cdot C\cdot (\varphi(x_r)-\varphi^*),$$
	holds for all $k\geq 1$, where we set $C:=\frac{2}{{L\beta(1-\sqrt{1-(1-q)^{-1}})^2}}$.
\end{theorem}
\begin{proof}
	Strong regularity of the subdifferential in Theorem \ref{thm:eqv_many}, in particular, implies that the subdifferential mapping $\partial \varphi$ is subregular at $(x^*,0)$ with constant $1/\alpha$. Theorem~\ref{thm:main_prop} then implies that the error bound condition holds near $x^*$ with constant  $\frac{5}{\alpha}+\frac{2}{L\beta}$. \Cref{thm:lin_conv_bad} then shows that the sequence $x_k$ converges to $x^*$. Consequently for all large indices $k$, the ratios $\frac{\|x_k-x^*\|}{\|\mathcal{G}_t(x_k)\|}$ are bounded by $\frac{5}{\alpha}+\frac{2}{L\beta}$. The inequality~\eqref{eqn:nat_conv} immediately yields the claimed Q-linear rate. The R-linear rate follows easily by a standard argument, as in the proof of Theorem~\ref{thm:lin_conv_bad}. 
\end{proof}

In summary, Theorem~\ref{thm:nat_rate} shows that if the prox-linear method is initialized sufficiently close to a stable strong local minimizer $x^*$ of $\varphi$ with constant $\alpha$, then the function values converge at a linear rate on the order of $1-\frac{\alpha}{L\beta}$.
This is in contrast to the  slower rate $1-(\frac{\alpha}{L\beta})^2$ established in Theorem~\ref{thm:lin_conv_bad} under the weaker condition that $\partial \varphi$ is  subregular at $(x^*,0)$ with constant $1/\alpha$.

\section{Proximal gradient method without convexity}\label{sec:prox_wo_con}
In this section, we revisit the proximal gradient method, discussed in \cref{sec:lin_con} in absence of convexity in $f$. To this end, consider the optimization problem
$$\min_{x\in\R^n} \varphi(x):=f(x)+g(x),$$
where $f\colon\R^n\to \R$  is a $C^1$-smooth function with $\beta$-Lipschitz gradient and $g\colon\R^n\to\overline{\R}$ is a closed convex function.
Observe that the proximal gradient method:
$$x_{k+1}=x_k-t\mathcal{G}_t(x_k),$$
with the gradient mapping $$\mathcal{G}_t(x):=t^{-1}\Big(x-\prox_{t g}(x-t\nabla f(x))\Big),$$
is simply an instance of the prox-linear algorithm (Algorithm~\ref{algo:prox_lin_simple}) applied to the function $\varphi=g+\textrm{Id}\circ f$. Here  $\textrm{Id}$ is the identity map on $\R$. Hence all the results of \cref{sec:prox_lin_alg,sec:conv_rate_tilt}  apply immediately with $L=1$. The arguments in this additive case are much simpler, starting with the fact that Ekeland's variational principle is no longer required to prove that subdifferential subregularity is equivalent to the error bound property. 

Indeed, the key perturbation result Theorem \ref{thm:perturb} simplifies drastically: in the notation of the theorem, we can set $\hat x:=x^t$. Then the optimality conditions for the proximal subproblem 
$$\mathcal{G}_t(x)\in \nabla f(x)+\partial g(x^t)$$
immediately imply the slightly improved estimate in Theorem~\ref{thm:perturb}:
\begin{equation}\label{eqn:imp_est}
\dist(0,\partial \varphi(x^t))\leq (1+\beta t)\|\mathcal{G}_t(x)\|.
\end{equation}
As in Theorem \ref{thm:main_prop}, consider now the two properties:
\begin{enumerate}[(i)]
	\item\label{eqn:subreg_sub33} the subdifferential $\partial \varphi$ is subregular at $(\bar x,0)$ with constant $l$.
	\item\label{eqn:subreg_prox33} the prox-gradient mapping $\mathcal{G}_t(\cdot)$ is subregular at $(\bar x,0)$ with constant $\hat l$.
\end{enumerate}
The same argument as in Theorem \ref{thm:main_prop} shows that if $(\ref{eqn:subreg_sub33})$ holds, then $(\ref{eqn:subreg_prox33})$ holds with $\hat{l}=l(\beta t+1)+t$. Conversely if $(\ref{eqn:subreg_prox33})$ is valid, then $(\ref{eqn:subreg_sub33})$ holds with $l=2\hat{l}$.

Next, we reevaluate convergence guarantees under tilt-stability. To this end, suppose that one of the equivalent conditions in Theorem \ref{thm:eqv_many} holds. Set for simplicity $t=\beta^{-1}$ and let $v$ be the minimal norm element of $\partial \varphi(x^t)$. We then deduce for all points $x$ and $y$ near $x^*$ the inequality
$$\varphi(y)\geq \varphi(x^t)+\langle v,y-x^t\rangle.$$
Plugging in $y=x^*$ yields
\begin{align*}
\varphi(x^t)-\varphi(x^*)&\leq \|\mathcal{G}_t(x)\|^2\left(\frac{\|v\|\cdot\|x^*-x^t\|}{\|\mathcal{G}_t(x)\|^2}\right).
\end{align*}
Appealing to Theorem \ref{thm:eqv_many} and the equivalence of subdifferential subregularity and the error bound property above, we deduce
$\frac{\|x^*-x^t\|}{\|\mathcal{G}_t(x)\|}\leq \frac{\|x^*-x\|}{\|\mathcal{G}_t(x)\|}+t\leq \alpha^{-1}(\beta t+1)+2t$
Taking into account the inequality \eqref{eqn:imp_est}, we obtain
\begin{align*}
\varphi(x^t)-\varphi(x^*)&\leq \|\mathcal{G}_t(x)\|^2\left(1+\beta t)(\alpha^{-1}(\beta t+1)+2t)\right)\\
&\leq 8\beta (\varphi(x)-\varphi(x^{t}))\left(\alpha^{-1}+\beta^{-1}\right).
\end{align*}
where the last inequality follows from \eqref{eqn:descent2}.
Trivial algebraic manipulations then yield the Q-linear rate of convergence 
$$\varphi(x_{k+1})-\varphi(x^*)\leq \left(1-\frac{1}{9+8\beta/\alpha}\right)(\varphi(x_k)-\varphi(x^*))$$
for all sufficiently large indices $k$. As an aside, this estimate slightly improves on the constants appearing in Theorem \ref{thm:nat_rate} for this class of problems.


\section{The proximal subproblem in full generality}\label{prox_lin_full_gen}
In this final section, we build on the convex composite framework explored in \cref{sec:prox_lin_alg,sec:conv_rate_tilt,sec:prox_wo_con} by dropping convexity and finite-valued-ness assumptions.
Specifically, we begin in great generality with the problem
\begin{equation}\label{eqn:composite_gen}
\min_{x}~\varphi(x):=g(x)+h(c(x)),
\end{equation}
where $c\colon\R^n\to\R^m$ is a $C^1$-smooth mapping and $g, h\colon\R^n\to\overline{\R}$ are merely closed functions.
As in \cref{sec:prox_lin_alg}, define for any points $x,y\in\R^n$ the linearized function 
$$\varphi(x;y):=g(y)+h\big(c(x)+\nabla c(x)(y-x)\big),$$
and for any real $t>0$, the quadratic perturbation
$$\varphi_{t}(x;y):=\varphi(x;y)+\frac{1}{2t}\|x-y\|^2.$$
It is tempting to simply apply the prox-linear algorithm directly to this setting by iteratively solving the proximal subproblems $\min_{y} \varphi_t(x;y)$ for appropriately chosen reals $t>0$. 
This naive strategy is fundamentally flawed.
There are two main difficulties. First, in contrast to the previous sections, the functions $\varphi(x;\cdot)$ and $\varphi_t(x;\cdot)$ are typically nonconvex. Hence when solving the  subproblems, one must settle only for ``stationary points'' of $\varphi_t(x;\cdot)$. Second, and much more importantly, the iterates generated by such a scheme can quickly yield infeasible proximal subproblems, thereby stalling the algorithm. Designing a variant of the prox-linear algorithm that overcomes the latter difficulty is an ongoing research direction and will be pursued in future work. Nonetheless, it is intuitively clear that the proximal subproblems may still enter the picture. In this section, we show that the equivalence between the ``error bound property'' and subdifferential subregularity in Theorem~\ref{thm:main_prop} extends to this broader setting. The technical content of this section is based on a careful mix of variational analytic techniques, which we believe is of independent interest.

For simplicity, we will assume $g=0$ throughout, that is we consider the problem 
\begin{equation}\label{targ:simpl}
\min_x~ \varphi(x)=h(c(x)),
\end{equation}
where $c\colon\R^n\to\R^m$ is $C^1$-smooth and $h\colon\R^n\to\overline{\R}$ is closed. 
The assumption $g=0$ is purely for convenience: all results in this section extend verbatim to the more general setting with identical proofs. As alluded to above, we will be interested in ``stationary points''  of nonsmooth and nonconvex functions $\varphi$ and $\varphi_t(x;\cdot)$. 
To make this notion precise, we appeal to a central variational analytic construction, the subdifferential; see e.g. \cite{VA,Mord_1,ioffe_survey}.

\begin{definition}[Subdifferentials and stationary points]{\hfill \\ }
	{\rm Consider a closed function $f\colon\R^n\to\overline{\R}$ and a point $\bar{x}$ with $f(\bar{x})$ finite. 
		\begin{enumerate}
			\item The {\em proximal subdifferential} of $f$ at $\bar{x}$, denoted 
			$\partial_p f(\bar{x})$, consists of all  $v \in \R^n$ for which there is a neighborhood $\mathcal X$ of $\bar x$ and a constant $r >0$ satisfying $$f(x)\geq f(\bar{x})+\langle v,x-\bar{x} \rangle -\frac{r}{2}\|x-\bar{x}\|^2 \qquad\textrm{ for all }x\in \mathcal X.$$ 
			\item The {\em limiting subdifferential} of $f$ at $\bar{x}$, denoted $\partial f(\bar{x})$, consists of all vectors $v\in\R^n$ for which there exist sequences $x_i\in\R^n$ and $v_i\in\partial_p  f(x_i)$ with $(x_i,f(x_i),v_i)$ converging to $(\bar{x},f(\bar{x}),v)$.
			\item The {\em horizon subdifferential} of $f$ at $\bar{x}$, denoted $\partial^{\infty} f(\bar{x})$, consists of all vectors $v\in\R^n$ for which there exist points $x_i\in\R^n$, vectors $v_i\in\partial  f(x_i)$, and real numbers $t_i\searrow 0$ with $(x_i,f(x_i),t_iv_i)$ converging to $(\bar{x},f(\bar{x}),v)$.
		\end{enumerate}
		We say that $\bar{x}$ is a {\em stationary point} of $f$ whenever the inclusion $0\in\partial f(\bar{x})$ holds.}
\end{definition}

Proximal and limiting subdifferentials of a convex function coincide with the usual convex subdifferential, as defined in \cref{sec:prelim}. The horizon subdifferential plays an entirely different role, detecting horizontal normals to the epigraph of the function. For example, a closed function $f\colon\R^n\to\overline \R$ is locally Lipschitz continuous around $\bar x$ if and only if $\partial^{\infty} f(\bar{x})=\{0\}$. Moreover, the horizon subdifferential plays a decisive role in establishing calculus rules and in stability analysis. 
We introduce the following notation to help the exposition.
\begin{definition}[Transversality]
	{\rm Consider a $C^1$-smooth mapping $c\colon\R^n\to\R^m$, a function $h\colon \R^m\to\overline{\R}$, and a point $\bar x$ with $h(c(\bar x))$ finite. We say that {\em $c$ is  transverse to $h$ at $\bar{x}$}, denoted $c\trans_{\bar x} h$,
		if the condition holds:
		\begin{equation}\label{eqn:transv_main}
		\partial^{\infty}h(c(\bar x))\cap {\rm Null}(\nabla c(\bar x)^*)=\{0\},
		\end{equation}
		If this condition holds with $h$ an indicator function of a set $Q$, then we say that {\em $c$ is transverse to $Q$ at $\bar x$}, and denote it by $c\trans_{\bar x} Q$.
	}	
\end{definition}

Transversality unifies both the classical notion of ``transverse intersections'' in differential manifold theory (e.g. \cite[Section 6]{Lee}) and the Mangasarian-Fromovitz constraint qualification in nonlinear programming (e.g. \cite[Example 9.44]{VA}, \cite[Example 4D.3]{imp}).  
Whenever a $C^1$-smooth mapping $c$ is transverse to a closed function $h$ at $\bar x$, the key inclusion 
$$\partial \varphi(\bar x)\subseteq \nabla c(\bar x)^*\partial h(c(\bar x))\qquad\textrm{ holds},$$
resembling the usual chain rule for smooth compositions. Equality is valid under additional assumptions, such as that $\partial_p h(c(\bar x))$ and $\partial h(c(\bar x))$ coincide for example.

We now define the {\em stationary point map}
$$\mathcal{S}_t(x)=\{z\in \R^n: z  \textrm{ is a stationary point of } \varphi_{t}(x,\cdot)\}$$
and the {\em prox-gradient mapping}
$$\mathcal{G}_t(x)=t^{-1}\big(x-S_t(x)\big).$$
Notice that both $\mathcal{S}_t$ and $\mathcal{G}_t$ are now set-valued operators. Our goal is to relate subregularity of $\mathcal{G}_t$ to subregularity of the subdifferential $\partial \varphi$ itself. The general trend of our arguments follows that of \cref{sec:prox_lin_alg}. There are important difficulties, however, that must be surmounted. An immediate difficulty is that since $h$ is possibly infinite-valued, it is not possible to directly relate the values $\varphi(y)$ and $\varphi(x;y)$, as in inequality \eqref{eqn:upper-lower} -- the starting point of the analysis in \cref{sec:prox_lin_alg}. For example, $\varphi(y)$ can easily be infinite while  $\varphi(x;y)$ is finite, or vice versa. The key idea is that such a comparison is possible if we allow a small perturbation of the point $y$. The following section is dedicated to establishing this result (Theorem \ref{thm:translate}).

\subsection{The comparison theorem}
We will establish Theorem \ref{thm:translate} by appealing to the fundamental relationship between transversality \eqref{eqn:transv_main}, metric regularity of constraint systems, and stability of metric regularity under linear perturbations \cite{radius}. We begin by recalling the concept of metric regularity --  a uniform version of subregularity (Definition \ref{defn:subreg}). For a discussion on the role of metric regularity in nonsmooth optimization, see for example \cite{imp,ioffe_survey}.

\begin{definition}[Metric regularity]
	{\rm	A set-valued mapping $G\colon\R^n\rightrightarrows \R^m$ is {\em metrically regular around}  $(\bar{x},\bar{y})\in \gph G$ with constant $l>0$ if there exists a neighborhood $\mathcal{X}$ of $\bar{x}$ and a neighborhood $\mathcal{Y}$ of $\bar y$ satisfying 
		$$\dist\left(x; G^{-1}(y)\right)\leq l\cdot\dist\left( y; G(x)\right)\qquad\textrm{for all }x\in \mathcal X,\, y\in \mathcal{Y}.$$}
\end{definition}

Metric regularity, unlike subregularity, is stable under small linear perturbations. The following is a direct consequence of the proof of \cite[Theorem 3.3]{radius}, as described in \cite[Theorem 4.2]{prx_lin}. 
\begin{theorem}[Metric regularity under perturbation]\label{thm:unimet_reg}
	Consider a closed set-valued mapping $G\colon\R^n\rightrightarrows\R^m$ that is metrically regular around $(\bar{x},\bar{y})\in \gph G$.	Then there exist  constants $\epsilon,\gamma>0$ and neighborhoods $\mathcal{X}$ of $\bar x$ and $\mathcal Y$ of $\bar y$, so that the estimate 
	$$\dist(x,(G+H)^{-1}(y))\leq \gamma\cdot \dist(y, (G+H)(x))\qquad\textrm{ holds}$$
	for any affine mapping $H$ with $\|\nabla H\|< \epsilon$, any $x\in \mathcal X$, and any $y\in H(\bar{x})+\mathcal Y$.
\end{theorem}

The following consequence for stability of constraint systems is now immediate. For any set $Q\subset\R^n$ and a point $x\in Q$, we define the {\em limiting normal cone} by $N_Q(x):=\partial \delta_Q(x)$. 
\begin{corollary}[Uniform metric regularity of constraint systems]\label{cor:uni_met_constr} {\hfill \\ }
	Fix a $C^1$-smooth mapping $F\colon\R^n\to\R^m$ and a closed set $Q\subset \R^m$, and consider the constraint system
	$$F(x)\in Q.$$
	Suppose that $F$ is transverse to $Q$ at some point $\bar x$, with $F(\bar{x})\in Q$.	
	Then there are constants $\epsilon,\gamma>0$ and a neighborhood $\mathcal X$ of $\bar x$ so that for any $x\in \mathcal X$ and any affine mapping $H$ with $\|\nabla H\|< \epsilon$ and $\|H(\bar{x})\|<\epsilon$, there exists a point $z$ satisfying
	$$(F+H)(z)\in Q\qquad \textrm{ and }\qquad \|x-z\|\leq \gamma \cdot \dist\Big((F+H)(x);Q\Big).$$
\end{corollary}
\begin{proof}
	Consider the set-valued mapping 
	$G(x):=F(x)-Q$. It is well-known (e.g. \cite[Example 9.44]{VA}) that the condition
	$F\trans_{\bar x} Q$ is equivalent to $G$ being metrically regular around $(\bar{x},0)$. 
	Appealing to Theorem~\ref{thm:unimet_reg}, we deduce
	that there are constants $\epsilon,\gamma>0$ and neighborhoods $\mathcal X$ of $\bar x$ and $\mathcal Y$ of $0$, so that for any $x\in \mathcal X$ and $y\in H(\bar{x})+\mathcal Y$ and any affine mapping $H$ with $\|\nabla H\|< \epsilon$, there exists a point $z$ satisfying
	$$F(z)+H(z)\in y+Q\qquad \textrm{ and }\qquad \|x-z\|\leq \gamma\, \dist(F(x)+H(x),y+Q).$$
	In light of the assumption $\|H(\bar{x})\|< \epsilon$, decreasing $\epsilon$, we can be sure that $-H(\bar{x})$ lies in $\mathcal Y$ and hence we can set $y:=0$ above. The result follows immediately.
\end{proof}

%
%
%

Armed with Corollary \ref{cor:uni_met_constr}, we can now prove the main result of this section, playing the role of inequalities \eqref{eqn:upper-lower}.


\begin{theorem}[Comparison inequalities]\label{thm:translate} 
	Consider the composite problem \eqref{targ:simpl} satisfying $c\trans_{\bar x}h$ for some point $\bar x\in \dom \varphi$, around which
	$\nabla c$ is Lipschitz continuous.
	Then there exist $\epsilon,\gamma>0$ and a neighborhood $\mathcal X$ of $\bar{x}$ such that the following hold.
	\begin{enumerate}
		\item For any two points $x,y\in \mathcal X$ with $| \varphi(y)-\varphi(\bar{x})|<\epsilon$, there exists a point $y^-$ satisfying
		$$\|y-y^-\|\leq \gamma \|y-x\|^2 \qquad \textrm{and}\qquad \varphi(x;y^-)\leq \varphi(y)+\gamma\|y-x\|^2.$$
		\item\label{claim:2_ineq} For any two points $x,y\in \mathcal X$ with $|\varphi(x;y)-\varphi(\bar{x})|<\epsilon$, there exists a point $y^+$ satisfying
		$$\|y-y^+\|\leq \gamma \|y-x\|^2 \qquad \textrm{and}\qquad \varphi(y^+)\leq  \varphi(x;y) +\gamma\|y-x\|^2.$$
	\end{enumerate}
\end{theorem}
\begin{proof}
	The second claim appears as \cite[Theorem 4.6]{prx_lin}.\footnote{In \cite[Theorem 4.6]{prx_lin}, it is assumed that $c$ is $C^2$-smooth; however, it is easy to verify from the proof that the same result holds if $\nabla c$ is only locally Lipschitz continuous around $\bar x$.} To see the first claim, for any point $x$ define the affine mapping $L_x(z)=c(x)+\nabla c(x)(z-x)$.
	Consider now the affine mapping $F\colon\R^{n+1}\to \R^{m+1}$ given by $$F(z,t):=(c(\bar{x})+\nabla c(\bar{x})(z-\bar{x}),t).$$ Then the transversality condition $c\trans_{\bar x}h$, along with the epigraphical characterization of subgradients \cite[Theorem 8.9]{VA}, implies
	that $F$ is transverse to $\epi h$ at $(\bar x,\varphi(\bar x))$.
	%
	Hence we can apply Corollary \ref{cor:uni_met_constr} with $Q:=\epi h$. Let $\epsilon,\gamma >0$ and $\mathcal X$ be the resulting constants and a neighborhood of $(\bar x,\varphi(\bar x))$, respectively.  Define now the affine mappings $H_x(z,t):=(L_{x}(z)-L_{\bar{x}}(z),0)$. Observe that for all $x$ sufficiently close to $\bar x$, the inequalities $\|\nabla H_x\|<\epsilon$ and $\| H_x(\bar x)\|<\epsilon$ hold.

	We deduce that for all pairs $(y,\varphi(y))$ sufficiently close to $(\bar x,\varphi(\bar x))$,  
	and for all points $x$ sufficiently close to $\bar x$, there exists a pair $(y^-,t)$ satisfying
	$$(L_x(y^-),t)=(F+H_x)(y^-,t)\in \epi h$$
	and 
	\begin{align*}
	\|(y,\varphi(y))-(y^-,t)\|&\leq \gamma\, \dist\Big((F+H_x)(y,\varphi(y));\epi h\Big)\\
	&=\gamma\, \dist\Big((L_x(y),\varphi(y));\epi h\Big)\\
	&\leq \gamma \|L_x(y)-c(y)\|\\
	&\leq \frac{\gamma \beta}{2}\|x-y\|^2,
	\end{align*}
	where $\beta$ is a Lipschitz constant of $\nabla c(\cdot)$ on a neighborhood of $\bar{x}$.
	Hence we deduce 
	$$\varphi(x;y^-)=h(L_{x}(y^-))\leq t\leq \varphi(y)+ \frac{\gamma\beta}{2}\|x-y\|^2$$
	and
	$\|y-y^-\|\leq \frac{\gamma\beta}{2}\|x-y\|^2$, as claimed. 
\end{proof}

\subsection{Prox-regularity of the subproblems}
The final ingredient we need to study subregularity of the mapping $\mathcal{G}_t$ is the notion of prox-regularity. The idea is that we must focus on functions whose limiting subgradients yields uniformly varying  quadratic minorants. These are  the prox-regular functions introduced in \cite{prox_reg}.

\begin{definition}[Prox-regularity]
	{\rm
		A closed function $f\colon\R^n\to\overline \R$	is {\em prox-regular} at $\bar{x}$ for $\bar{v}\in \partial f(\bar{x})$ if there exist neighborhoods $\mathcal X$ of $\bar x$ and $\mathcal V$ of $\bar v$, along with constants $\epsilon,r>0$ so that the inequality 
		$$f(y)\geq f(x) +\langle v,y-x\rangle-\frac{r}{2}\|y-x\|^2,$$
		holds for all $x,y\in \mathcal X$ with $|f(x)-f(\bar x)|< \epsilon$, and for every subgradient $v\in \mathcal V\cap \partial f(x)$.
	}
\end{definition}

Prox-regular functions are common in nonsmooth optimization, encompassing for example all 
$C^1$-smooth functions with Lipschitz gradients and all closed convex functions. More generally, the authors of \cite{calc_prox} showed that the composite function $\varphi=h\circ c$ is 
prox-regular at $\bar x$ for $\bar v\in \partial\varphi(\bar x)$ whenever $c$ is  $C^1$-smooth with a Lipschitz gradient, 
$c$ is transverse to $h$ at $\bar x$, and $h$ is prox-regular at $c(\bar x)$ for every vector $w\in \partial h(c(\bar x))$ satisfying $\bar v=\nabla c(\bar x)^*w$. The following proposition shows that under these conditions, the linearized functions $\varphi(x,\cdot)$  are also prox-regular, uniformly in $x$.

\begin{proposition}[Uniform prox-regularity of the subproblems]\label{prop:uni_prox} {\hfill \\ }
	Consider the composite problem \eqref{targ:simpl} satisfying $c\trans_{\bar x}h$ for some point $\bar x\in \dom \varphi$.
	Then there exists a neighborhood $\mathcal X$ of $\bar x$ and a constant $\epsilon>0$ so that 
	the affine functions $z\mapsto c(x)+\nabla c(x)(z-x)$ are tranverse to  
	$h$ at $z$, 	
	%
	for any $x,z\in \mathcal X$ with $|\varphi(x;z)-\varphi(\bar x)|<\epsilon$.
	
	Consider a vector $\bar{v}\in \partial \varphi(\bar x)$, and suppose also that $\nabla c$ is Lipschitz continuous around $\bar x$ and that $h$ is prox-regular at $c(\bar x)$ for every subgradient $w\in \partial h(c(\bar x))$ satisfying $\bar v=\nabla c(\bar x)^*w$. Then 
	the linearized functions $\varphi(x;\cdot)$ are  prox-regular at $\bar x$ for $\bar v$ uniformly in $x$ in the following sense.
	After possibly shrinking $\mathcal{X}$ and $\epsilon >0$, there exists a neighborhood $\mathcal V$ of $\bar v$ and a constant $\gamma>0$ so that the inequality
	\begin{equation}\label{eqn:uni_prox_super}
	\varphi(x;y)\geq \varphi(x;z)+\langle v,y-z \rangle -\frac{\gamma}{2}\|y-z\|^2
	\end{equation}
	holds for any $x,y,z\in \mathcal X$ with $|\varphi(x;z)-\varphi(\bar x)|< \epsilon$, and for every subgradient $v\in \mathcal V\cap  \partial_z \varphi (x;z)$. 
\end{proposition}
\begin{proof}
	For the sake of contradiction, suppose there exist sequences $x_i\to\bar x$ and $z_i\to\bar x$ along with unit vectors $w_i\in \partial^{\infty}h\big(c(x_i)+\nabla c(x_i)(z_i-x_i)\big)\cap {\rm Null}(\nabla c(x_i)^*)$, so that the values $h\big(c(x_i)+\nabla c(x_i)(z_i-x_i)\big)$ tend to $h(c(\bar x))$.
	Passing to a subsequence, we may suppose that $w_i$	converge to some unit vector $w\in \partial^{\infty}h(c(\bar x))\cap{\rm Null}(\nabla c(\bar x)^*)$, a contradiction. Hence the transversality claim holds.
	
	By the prox-regularity assumption on $h$, for every subgradient $w\in \partial h(c(\bar x))$ satisfying $\bar v= \nabla c(\bar x)^* w$, there exist constants $\delta_w,r_w>0$ such that the inequality
	$$h(\xi_2)\geq h(\xi_1)+\langle \eta,\xi_2 -\xi_1\rangle-\frac{r_w}{2}\|\xi_2-\xi_1\|^2$$
	holds for all $\xi_1,\xi_2\in B_{\delta_{w}}(c(\bar x))$ with $|h(\xi_1)-\varphi(\bar x)|<\delta_w$, and for any subgradient $\eta\in \partial h(\xi_1)$ with $\|\eta-w\|<\delta_w$. We claim that there exist uniform constants $\delta,r>0$ (independent of $w$) so that the inequality 
	\begin{equation}\label{eqn:uni_prox}
	h(\xi_2)\geq h(\xi_1)+\langle \eta,\xi_2 -\xi_1\rangle-\frac{r}{2}\|\xi_2-\xi_1\|^2 \end{equation}
	holds for all $\xi_1,\xi_2\in B_{\delta}(c(\bar x))$ with $|h(\xi_1)-\varphi(\bar x)|<\delta$, and for any subgradient $\eta\in \partial h(\xi_1)$ with $\|\nabla c(\bar x)^*\eta-\bar v\|<\delta$. To see that, suppose otherwise and consider sequences $\xi_1,\xi_2\to c(\bar x)$ with $h(\xi_1)\to \varphi(\bar x)$, and a sequence $\eta\in \partial h(\xi_1)$ with $\nabla c(\bar x)^*\eta\to\bar v$ and so that the fractions $\frac{h(\xi_2)- (h(\xi_1)+\langle \eta,\xi_2 -\xi_1\rangle)}{\|\xi_2-\xi_1\|^2}$ are not lower-bounded. The transversality condition \eqref{eqn:transv_main} immediately implies that the vectors $\eta$ are bounded and hence we can assume that $\eta$ converge to some vector $w\in \partial h(c(\bar x))$ with $\nabla c(x)^*w=\bar v$. This immediately yields the contradiction $\frac{h(\xi_2)- (h(\xi_1)+\langle \eta,\xi_2 -\xi_1\rangle)}{\|\xi_2-\xi_1\|^2}\geq -r_w$ for the tails of the sequences.

	Fix now points $x,y,z\in \mathcal X$ with $|\varphi(x;z)-\varphi(\bar x)|< \epsilon$.
	Consider a subgradient $v\in \partial_z \varphi (x;z)$. Since we have already proved that the affine function $c(x)+\nabla c(x)(\cdot-x)$ is transverse to $h$ at $z$, we may write 
	$v=\nabla c(x)^*\eta$ for some subgradient $\eta\in \partial h(c(x)+\nabla c(x)(z-x))$. Define now the points $\xi_1:=c(x)+\nabla c(x)(z-x)$ and $\xi_2:=c(x)+\nabla c(x)(y-x)$. 
	Shrinking $\mathcal{X}$ and $\epsilon>0$, we can ensure  $\xi_1,\xi_2\in B_{\delta}(c(\bar x))$ with $|h(\xi_1)-\varphi(\bar x)|<\delta$. Moreover, if $v$ is sufficiently close to $\bar v$  we can ensure $\|\nabla c(\bar x)^*\eta-\bar v\|<\delta$. Hence we may apply inequality \eqref{eqn:uni_prox}, yielding
	$$\varphi(x;y)=h(\xi_2)\geq h(\xi_1)+\langle \eta,\nabla c(x)(y-z)\rangle -\frac{r}{2}\|\nabla c(x)(y-z)\|^2,$$ and therefore
	$\varphi(x;y)\geq \varphi(x;z)+\langle v,y-z\rangle -\frac{rL^2}{2}\|y-z\|^2$, where $L:=\max_{x\in \mathcal X} \|\nabla c(x)\|$. The result follows.
\end{proof}

We are ready to prove out main tool generalizing Theorem \ref{thm:perturb}.

\begin{theorem}[Prox-gradient and near-stationarity]\label{thm:perturb2}
	Consider the composite problem \eqref{targ:simpl} satisfying $c\trans_{\bar x}h$ for some point $\bar x$ with $0\in \partial \varphi(\bar x)$.
	Suppose that $\nabla c$ is locally Lipschitz around $\bar x$ and that $h$ is prox-regular at $c(\bar x)$ for every subgradient $w\in \partial h(c(\bar x))$ satisfying $0=\nabla c(\bar x)^*w$. Then there are constants $\gamma,\epsilon,a,b>0$ and
	a neighborhood $\mathcal X$ of $\bar x$ such that for any $t>0$, $x\in \mathcal X$, and  $x^t\in \mathcal X\cap S_t(x)$ with $|\varphi(x,x^t)-\varphi(\bar x)|< \epsilon$, there exists a point  $\hat x$ 
	satisfying the properties
	\begin{enumerate}[(i)]
		\item (point proximity) $\quad\|x^t-\hat x\|\leq \big(1+\gamma\|x^t-x\|\big)\cdot \|x^t-x\|$,
		\item (functional proximity) $\varphi(\hat x)\leq \varphi(x;x^t) +(a+b/t)\cdot\|x^t-x\|^2$,		
		\item\label{it:near_stat2} (near-stationarity) $\quad\dist\left(0;\partial \varphi(\hat x)\right)\leq (a+b/t)\cdot\|x^t-x\|$.
	\end{enumerate}

\end{theorem}
\begin{proof}
	Fix a neighborhood $\mathcal X$ of $\bar x$ and constants $\epsilon,\gamma>0$ given by Theorem \ref{thm:translate}. Shrinking $\mathcal X$ we can assume that $\mathcal X$ is closed and has diameter smaller than one (for simplicity).
	Suppose $x$ lies in $\mathcal X$ and fix a point $x^t\in \mathcal X\cap S_t(x)$ with $|\varphi(x,x^t)-\varphi(\bar x)|< \epsilon$. Then for any $y\in \mathcal X$ with $|\varphi(y)-\varphi(\bar x)|< \epsilon$,  there exists a point $y^-$ satisfying 
	\begin{equation}\label{eqn:approx_min}
	\|y-y^-\|\leq \gamma \|y-x\|^2 \quad\textrm{ and }\quad 
	\varphi(y)+\gamma\|y-x\|^2\geq \varphi(x;y^-).
	\end{equation}
	Appealing to Proposition~\ref{prop:uni_prox}, we can be sure there is a constant $r>0$ so that for all $x,x^t,y\in \mathcal X$ with  $|\varphi(y)-\varphi(\bar x)|< \epsilon$, we have
	\begin{equation}\label{eqn:ineq_phi}
	\begin{aligned}
	\varphi(x;y^-)&\geq  \varphi(x,x^t)+\langle t^{-1}(x-x^t), y^--x^t\rangle -\frac{r}{2}\|y^--x^t\|^2\\
	&=\varphi_t(x,x^t)-\frac{1}{2t}\Big(\|x-x^t\|^2 -2\langle x-x^t, y^--x^t\rangle +rt\|y^--x^t\|^2\Big)\\
	&=\varphi_t(x,x^t) -\frac{1}{2t}\Big(\|x-y^-\|^2+(rt-1)\|y^--x^t\|^2\Big),
	\end{aligned}
	\end{equation}
	where the last inequality follows by completing the square. 
	On the other hand, the triangle inequality, along with the assumption that the diameter of $\mathcal{X}$ is smaller than one, yields
	$$\|x-y^-\|\leq (\gamma+1)\|x-y\| \quad \textrm{ and }\quad \|y^--x^t\|\leq (\gamma+1)\|x-y\|+\|x-x^t\|.$$
	
	Defining for notational convenience $q:=\gamma+1$ and $p:=\max\{0,rt-1\}$, we obtain the inequality
	$$\varphi(x;y^-)\geq \varphi_t(x,x^t)-\frac{q^2(1+p)}{2t}\|x-y\|^2  -\frac{p}{2t}\Big(\|x^t-x\|^2 +2q\max\{\|y-x\|,\|x^t-x\|\}^2\Big).$$
	Define now the function 
	\begin{align*}	
	\zeta(y):=\delta_{\mathcal{X}}(y)+&\varphi(y)+\Big(\gamma+\frac{q^2(1+p)}{2t}\Big)\|x-y\|^2 + \\
	&+\frac{p}{2t}\Big(\|x^t-x\|^2 +2q\max\{\|y-x\|,\|x^t-x\|\}^2\Big).
	\end{align*}	
	Taking into account \eqref{eqn:approx_min}, we deduce that the inequality
	$\zeta(y)\geq \varphi_t(x,x^t)$ holds for all $y\in \mathcal X$ with $|\varphi(y)- \varphi(\bar x)|< \epsilon$. Moreover, since $\varphi$ is closed, shrinking $\mathcal X$ we can assume $\varphi(y)\geq \varphi(\bar x)-\epsilon$ for all $y\in \mathcal X$. A trivial argument now shows that again shrinking $\mathcal X$, we can finally ensure $\zeta(y)\geq \varphi_t(x,x^t)$ for all $y\in \mathcal X$. 
	
	Taking into account the inequality $|\varphi(x,x^t)-\varphi(\bar x)|< \epsilon$ and  applying claim $\ref{claim:2_ineq}$ of Theorem~\ref{thm:translate}, a quick computation shows
	$$\zeta(x^{t+})-\zeta^*\leq l \cdot \|x^t-x\|^2,$$ 
	where we set 
	$$l:=2\gamma+\frac{q^2(1+p)+p(1+2q)-1}{2t},$$
	and $\zeta^*$ is the minimal value of $\zeta$.
	Define now the constants $\hat{\epsilon}:=l\|x^t-x\|^2$ and $\rho:=l\|x^t-x\|$.
	Applying Ekeland's variational principle, we obtain
	a point $\hat x$ satisfying the inequalities 
	$\zeta(\hat x)\leq \zeta(x^{t+})$ and $\|x^{t+}-\hat x\|\leq \hat{\epsilon}/\rho$, and the inclusion $0\in \partial \zeta(\hat x)+\rho {\bf B}$. The point proximity estimate is now immediate from the inequality $$\|x^{t}-\hat x\|\leq \|x^{t}-x^{t+}\|+\|x^{t+}-\hat x\|\leq \gamma \|x^t-x\|^2+\|x^t-x\|=(1+\gamma \|x^t-x\|)\|x^t-x\|.$$
	Using the inequalities $\varphi(\hat x)+\frac{p}{2t}\|x^t-x\|^2\leq\zeta(\hat x)\leq \zeta(x^{t+})$, we deduce 
	$$\varphi(\hat x)\leq \varphi(x, x^t)+\left(l-\frac{p-1}{2t}\right)\cdot\|x^t-x\|^2.$$
	Finally, we conclude the near-stationarity condition
	\begin{align*}
	\dist(0,\partial \varphi(\hat x))&\leq \rho+\left(\gamma+\frac{q^2(1+p)+2pq}{2t}\right)\|\hat x-x\|\\
	&\leq \left[l+(2+\gamma\|x^t-x\|)\left(\gamma+\frac{q^2(1+p)+2pq}{2t}\right)\right]\|x^t-x\|.
	\end{align*}	
	%
	%
	The result follows after noting the inequality $\frac{p}{t}\leq r$.
\end{proof}

We can now prove the main result of this section comparing subregularity of the subdifferential $\partial \varphi$ and of the prox-gradient mapping $\mathcal{G}_t$. Naturally, to make such a comparison precise, we must focus on the subgraphs of $\gph \partial \varphi$ and $\gph \mathcal{G}_t$ that arise from points $x$ and $x^t\in \mathcal{S}_t(x)$ at which the function values $\varphi(x)$ and $\varphi(x;x^t)$ are close to $\varphi(\bar x)$. 
In most important circumstances, nearness in the graphs, $\gph \partial \varphi$ and $\gph \mathcal{S}_t$, automatically implies nearness in function value.
To illustrate, consider the composite problem \eqref{targ:simpl}  satisfying
$c\trans_{\bar x}h$ for some point $\bar x$ with $0\in \partial \varphi(\bar x)$. It follows quickly from \cite[Example 13.30]{VA} that if $h$ is either convex or continuous on its domain, then $h$ is {\em subdifferentially continuous} at $(\bar x,0)$:
for any sequence $(x_i,v_i)\in \gph \partial \varphi$ converging to $(\bar x,0)$ the values $\varphi(x_i)$ converge to $\varphi(\bar x)$.  Similarly, it is easy to check that if $h$ is either convex or continuous on its domain, then for any sequence $(x_i,y_i)\in \gph \mathcal{S}_t$ converging to $(\bar x,\bar x)$ the values $\varphi(x_i;y_i)$ converge to $\varphi(\bar x)$.

When $h$ is not convex, nor is continuous on its domain, we must focus only on the relevant parts of the graphs, $\gph \partial \varphi$ and $\gph \mathcal{G}_t$.
This idea of a functionally attentive localization is not new, and goes back at least to \cite{prox_reg}.


\begin{definition}[$\varphi$ and $\varphi(\cdot,\cdot)$-attentive localizations]\label{defn:att_loc} 
	{\rm
		Consider the composite problem \eqref{targ:simpl} satisfying $c\trans_{\bar x}h$ for some point $\bar x$ with $0\in \partial \varphi(\bar x)$.
		\begin{enumerate}
			\item 
			A set-valued mapping $W\colon\R^n\rightrightarrows\R^n$ is a $\varphi$-{\em attentive localization} of the subdifferential $\partial \varphi$ around $(\bar x,0)$ if there exist neighborhoods $\mathcal{X}$ of $\bar x$ and $\mathcal V$ of $0$ and a real number $\epsilon>0$ so that 
			for any points $x\in \mathcal X$ and $v\in\mathcal{V}$ with $|\varphi(x)-\varphi(\bar x)|<\epsilon$, the equivalence $v\in \partial \varphi(x)\Longleftrightarrow v\in W(x)$ holds.
			\item  	A set-valued mapping $W\colon\R^n\rightrightarrows\R^n$ is a $\varphi(\cdot,\cdot)$-{\em attentive localization} of the stationary point map $\mathcal{S}_t$ around $(\bar x,\bar x)$ if there exist neighborhoods $\mathcal{X}$ of $\bar x$ and $\mathcal Y$ of $0$ and a real number $\epsilon>0$ so that 
			for any points $x\in \mathcal X$ and $y\in\mathcal{Y}$ with $|\varphi(x,y)-\varphi(\bar x)|<\epsilon$, the equivalence $y\in \mathcal{S}_t(x)\Longleftrightarrow y\in W(x)$ holds.
			\item A set-valued mapping $W\colon\R^n\rightrightarrows\R^n$ is a $\varphi(\cdot,\cdot)$-{\em attentive localization} of the prox-gradient mapping $\mathcal{G}_t$ around $(\bar x,0)$ if it can be written as $W=t^{-1}(I-\widehat{W}(x))$, where $\widehat{W}$ is a $\varphi(\cdot,\cdot)$-attentive localization of $\mathcal{S}_t$ around $(\bar x,\bar x)$.
		\end{enumerate}}
	\end{definition}


	The following is the main result of this section.
	
	\begin{theorem}
		[Subdifferential subregularity and the error bound property]\label{thm:nonconv_equiv} {\hfill \\ }
		Consider the composite problem \eqref{targ:simpl} satisfying $c\trans_{\bar x}h$ for some point $\bar x$ with $0\in \partial \varphi(\bar x)$.
		Suppose that $\nabla c$ is Lipschitz around $\bar x$ and that $h$ is prox-regular at $c(\bar x)$ for every subgradient $w\in \partial h(c(\bar x))$ satisfying $0=\nabla c(\bar x)^*w$. Consider the following two conditions:
		\begin{enumerate}[(i)]
			\item\label{it:1end} there exists a $\varphi$-attentive localization of $\partial \varphi$ that is metrically subregular at $(\bar x,0)$. 
			\item\label{it:2end} there exists a $\varphi(\cdot,\cdot)$-attentive localization of the mapping $\mathcal{G}_t$ that is metrically subregular at $(\bar x,0)$. 
		\end{enumerate}
		Then the implication $(\ref{it:1end})\Rightarrow (\ref{it:2end})$ always holds.
		Conversely, there exists a number ${\bar t\geq 0}$ so that for all $t \in (0,\bar t)$, the implication $(\ref{it:2end})\Rightarrow (\ref{it:1end})$ holds. 
		When $h$ is convex, the localizations are not needed: the following two conditions are equivalent for any $t>0$:
		\begin{enumerate}[(i')]
			\item  The subdifferential $\partial \varphi$ is metrically subregular at $(\bar x,0)$. 
			\item  The prox-gradient $\mathcal{G}_t$ is metrically subregular at $(\bar x,0)$. 
		\end{enumerate}
	\end{theorem}
	\begin{proof}
		Suppose there exists a $\varphi$-attentive localization $W$ of $\partial \varphi$ that is metrically subregular at $(\bar x,0)$ with constant $l$. Fix a neighborhood $\mathcal X$ of $\bar x$ and constants  $\gamma,\epsilon,a,b>0$  guaranteed to exist by Theorem \ref{thm:perturb2}. Then for any points $x\in \mathcal{X}$ and $x^t\in \mathcal{X}\cap\mathcal{S}_t(x)$ with $|\varphi(x,x^t)-\varphi(\bar x)|<\epsilon$, there exists a point $\hat x$ satisfying
		\begin{align}
		\|x^t-\hat x\|&\leq  \big(1+\gamma\|x^t-x\|\big)\cdot \|x^t-x\|,\\
		\varphi(\hat x)&\leq \varphi(x;x^t) +(a+b/t)\cdot\|x^t-x\|^2, \label{eqn:func_proximity}\\
		\dist (0;\partial \varphi(\hat x))&\leq (at+b)\cdot\|t^{-1}(x^t-x)\|.
		\end{align}
		Due to inequality \eqref{eqn:func_proximity}, the subdifferential $\partial \varphi(\hat x)$ coincides with the localization $W(\hat x)$ near the origin. Hence we deduce
		$$\dist(\hat x;W^{-1}(0))\leq l\cdot \dist (0;W(\hat x))=l\cdot \dist (0;\partial \varphi(\hat x))\leq l(at+b)\cdot\|t^{-1}(x^t-x)\|.$$
		On the other hand, we have 
		\begin{align*}
		\dist(\hat x,W^{-1}(0))&\geq \dist( x;W^{-1}(0))- \|x^t-x\|-\|x^t-\hat x\|,
		\end{align*}
		and therefore
		$$\dist( x;W^{-1}(0))\leq ((2+la+\gamma\|x^t-x\|)t+lb)\cdot\|t^{-1}(x^t-x)\|.$$
		Hence property $(\ref{it:2end})$ holds, as claimed.
		
		Next, we show the converse. Fix the neighborhoods $\mathcal{X},\mathcal{V}$ and the constants $\epsilon,\gamma$ guaranteed by Proposition \ref{prop:uni_prox}.
		After possibly shrinking $\mathcal{X}$, the local existence result \cite[Thorem 4.5(a)]{prx_lin} guarantees that there is a constant $\bar t$ so that provided $t<\bar t$, for any $x\in \mathcal{X}$ there exists a point
		$x^t\in \widehat{W}(x)$ so that the inclusion $x-x^t\in \mathcal{V}$ holds and we have $|\varphi(x;x^t)-\varphi(\bar x)|<\epsilon$. Henceforth  suppose that the map $W=t^{-1}(I-\widehat{W})$ is subregular at $(\bar{x},0)$ for some $t<\bar t$, where $\widehat{W}$ is a $\varphi(\cdot;\cdot)$-attentive localization of $\mathcal{S}_t$ around $(\bar x, \bar x)$.

		Fix a pair $(x,v)\in (\mathcal{X}\times \mathcal{V})\cap\gph \partial \varphi$ with $|\varphi (x)-\varphi(\bar x)|<\epsilon$. 
		Then $v$ lies in $\partial_z \varphi(x;z)$ with the choice $z:=x$, and hence setting $y:=x^t$ in \eqref{eqn:uni_prox_super} we deduce 
		\begin{align}
		\varphi(x,x^t)&\geq  \varphi(x)+\langle v, x^t-x \rangle-\frac{\gamma}{2}\|x^t-x\|^2.\label{eqn:little}
		\end{align}
		Appealing to \eqref{eqn:uni_prox_super} again with the choices $y=x$, $z=x^t$, and $t^{-1}(x-x^t)\in \partial_z \varphi(x;z)$, we obtain the inequality 
		\begin{equation}\label{eqn:reverse_ineq}
		\begin{aligned}
		\varphi(x)&\geq \varphi(x;x^t)+t^{-1}\langle x-x^t,x-x^t \rangle -\frac{\gamma}{2}\|x^t-x\|^2\\
		&=\varphi(x,x^t)+\left(t^{-1}-\frac{\gamma}{2}\right)\|x^t-x\|^2. 
		\end{aligned}
		\end{equation}
		Taking into account \eqref{eqn:little}, we deduce
		$\|v\|\cdot\|x^t-x\|\geq (t^{-1}-\gamma)\|x^t-x\|^2$. Choosing $t<\min\{\bar t, \gamma^{-1}\}$ so as to ensure $t^{-1}-\gamma>0$, we finally conclude 
		$$(t^{-1}-\gamma)^{-1}\cdot \|v\|\geq \|x^t-x\|\geq t\cdot \dist(0,W(x))\geq tl\cdot \dist(x,W^{-1}(0)),$$
		where  $l$ is the constant of subregularity of $W$ at $(\bar x,0)$. On the other hand, observe
		$z\in W^{-1}(0)\Leftrightarrow z=\widehat{W}(z)$.  
		Hence if $z\in  W^{-1}(0)$ is close to $\bar x$, then we can be sure that $\varphi(z)$ is close to $\varphi(\bar x)$. It follows that the set  $W^{-1}(0)$ coincides near $\bar x$ with $F^{-1}(0)$, where $F$ is some $\varphi$-attentive localization of $\partial \varphi$. Hence 
		Property~$(\ref{it:1end})$ holds, as claimed.
		
		Finally, when $h$ is convex, the functionally attentive localizations are not needed, as was explained prior to Definition \ref{defn:att_loc}.	
		Moreover, the inequalities \eqref{eqn:little} and \eqref{eqn:reverse_ineq} hold with $\gamma=0$ and according to \cite[Thorem 4.5(c)]{prx_lin}, we could have set $\bar t=+\infty$ at the onset. Consequently, the implication $(\ref{it:2end})\Rightarrow (\ref{it:1end})$ holds for arbitrary $t>0$, as claimed.
	\end{proof}

	To summarize, with Theorem~\ref{thm:nonconv_equiv} we have shown an equivalence between subregularity of the subdifferential $\partial\varphi$ and the error bound property, thereby extending Theorem \ref{thm:main_prop} to the case where $h$ need not be convex nor finite-valued, but merely prox-regular. We believe that this result can serve the same role as Theorem \ref{thm:main_prop} in \cref{sec:prox_lin_alg} for understanding linear convergence of proximal algorithms for composite problems  \eqref{eqn:composite_gen}.


	\subsection*{Acknowledgments}
	We thank Guoyin Li and Ting Kei Pong for a careful reading of an early draft of the manuscript, and for  insightful comments and suggestions.

\nopagebreak
\bibliographystyle{plain}
\bibliography{dima_bib}

\end{document}